\documentclass{amsart}     

\usepackage{graphicx}
\usepackage{color}						
\usepackage{amssymb,amsthm,amsfonts, mathrsfs}
\usepackage{bm}
\usepackage{yhmath}

\usepackage{pgfplots}
\usepgflibrary{arrows.meta}
\usetikzlibrary{patterns}
\usepgfplotslibrary{colormaps}
\usetikzlibrary{pgfplots.colormaps}
\usepgfplotslibrary{fillbetween}
\tikzset{> = {Stealth[inset=0pt]}}

\newtheorem{theorem}{Theorem}[section]
\newtheorem{lemma}[theorem]{Lemma}
\newtheorem{corollary}[theorem]{Corollary}
\newtheorem{proposition}[theorem]{Proposition}
\theoremstyle{definition}
\newtheorem{definition}[theorem]{Definition}
\theoremstyle{remark}
\newtheorem{remark}[theorem]{Remark}

\newcommand{\eps}{\varepsilon}

\newcommand{\N}{{\mathbb{N}}}
\newcommand{\Z}{{\mathbb{Z}}}

\newcommand{\R}{{\mathbb{R}}}

\newcommand{\from}{\colon}
\newcommand{\loc}{{\mathrm{loc}}}

\newcommand{\bx}{{\bm{x}}}
\newcommand{\by}{{\bm{y}}}
\newcommand{\be}{{\bm{e}}}
\newcommand{\bp}{{\bm{p}}}

\newcommand{\bnu}{{\bm{\nu}}}

\newcommand{\Rcal}{{\mathcal{R}}}
\newcommand{\Scal}{{\mathcal{S}}}
\newcommand{\Zcal}{{\mathcal{Z}}}

\DeclareMathOperator{\dist}{dist}
\DeclareMathOperator{\supp}{supp}
\DeclareMathOperator{\real}{Re}
\DeclareMathOperator{\area}{area}
\DeclareMathOperator{\diam}{diam}

\newcommand{\Xmin}{{X_{\min}}}

\newcommand{\Xmax}{{X_{\max}}}

\newcommand{\zeroset}{{\mathcal{Z}}}
\newcommand{\regset}{{\mathcal{R}}}
\newcommand{\singset}{{\mathcal{W}}}
\newcommand{\concomp}{D}
\newcommand{\np}{{n^*}}

\begin{document}                        

\title{Spiraling asymptotic profiles\\ of competition-diffusion systems}

\author{Susanna Terracini}\address{Dipartimento di Matematica ``Giuseppe Peano'', Universit\`a di Torino, Via Carlo Alberto 10, 10123 Torino, Italy}\email{susanna.terracini@unito.it}
\author{Gianmaria Verzini}\address{Dipartimento di Matematica, Politecnico di Milano, piazza Leonardo da Vinci 32, 20133 Milano, Italy}\email{gianmaria.verzini@polimi.it}
\author{Alessandro Zilio}\address{Universit\'e Paris Diderot, Sorbonne Paris Cit\'e, Laboratoire Jacques-Louis Lions (CNRS UMR 7598), 8 place Aur\'elie Nemours, 75205, Paris CEDEX 13, France}\email{azilio@math.univ-paris-diderot.fr}

\begin{abstract}
This paper describes the structure of the nodal set of segregation profiles arising in the singular limit of planar, stationary, reaction-diffusion systems with strongly competitive interactions of Lotka-Volterra type, when the matrix of the inter-specific competition coefficients is asymmetric and the competition parameter tends to infinity. Unlike the symmetric case, when it is known that the nodal set consists in a locally finite collection of curves meeting with equal angles at a locally finite number of singular points, the asymmetric case shows the emergence of spiraling nodal curves, still meeting at locally isolated points with finite vanishing order.
\end{abstract}

\maketitle

\section{Introduction}

This paper describes the structure of the nodal set of segregation profiles arising in the singular limit of planar, stationary, reaction-diffusion systems with strongly competitive interactions of Lotka-Volterra type, as and the competition parameter tends to infinity. This structure has been widely studied when the matrix of the inter-specific competition coefficients is symmetric in connection with either the free boundary of optimal partitions involving shape energies, or the singularities of harmonic maps with values in a stratified varyfold \cite{CoTeVe_AM2005, MR2529504, MR2644773, MR2483815, MR2984134}. For such problems, in the planar case, it is known that the nodal set consists in a locally finite collection of curves meeting with equal angles at a locally finite number of singular points. Our aim is to show that the effect of asymmetry of the inter-specific competition rates results in a dramatic change of the nodal pattern, now consisting of spiraling nodal curves, still meeting at locally isolated points with finite vanishing order. It has to be noticed that the asymmetry makes all the usual free boundary toolbox (Almgren and Alt-Caffarelli-Friedman monotonicity formul\ae, dimension estimates) unavailable and ad-hoc arguments have been designed. Finally, we point out that spiraling waves also occur in entirely different contexts of reaction-diffusion systems (cfr e.g. \cite{MR1952568,MR1438374,MR2318665}).

Lotka-Volterra type systems are the most popular mathematical models for the dynamics of many populations subject to spatial diffusion, internal reaction and either cooperative or competitive interaction.  Indeed, such models are associated with reaction-diffusion systems where the reaction is the sum of an intra-specific term, often expressed by logistic type functions, and an inter-specific interaction one, usually quadratic. The study of this
reaction-diffusion system has a long history and there exists a large literature on the subject. However, most of these works are concerned with the case of two species. As far as we know, the study in the case of many competing species has been
much more limited, starting from two pioneering papers by Dancer and Du \cite{MR1312772,MR1312773} in the 1990s, where the competition of three species were considered.

Let $\Omega$ be a domain of $\R^N$. We consider a system of $k$ non-negative densities, denoted by $u_1,\dots,u_k$, which are subject to diffusion, reaction and competitive interaction.
In the stationary case, the equations of the system take the form
\begin{equation}\label{eq:generalLV}
  -\Delta u_i = f_i(u_i)-\beta u_i \sum_{j\neq i} a_{ij} u_j
\quad \text{ in }\Omega,  \qquad i=1,\dots,k.
\end{equation}
Here $(a_{ij} )_{ij}$  is the matrix of the interspecific competition coefficients, with positive entries, and the parameter $\beta> 0$ measures the strength of the (competitive) interaction. Accordingly, we will distinguish between the \emph{symmetric} (i.e. when $a_{ij}=a_{ji}$, for every $i,j$) and  the \emph{asymmetric} case (i.e. when $a_{ij}\neq a_{ji}$ for some $i,j$). For concreteness we require the reaction terms $f_i$ to be locally Lipschitz, with $f_i(0)=0$,
even though specific results hold under less restrictive assumptions. One may consider also different diffusion coefficients $d_i>0$ on the left hand side of \eqref{eq:generalLV}. Nonetheless, in the stationary case, it is not restrictive to assume $d_i=1$ by a change of unknowns.

We will focus on the so called strong competition regime, that is when the parameter $\beta$ diverges to $+\infty$. In this case, the components satisfy uniform bounds in H\"older norms and converge, up to subsequences, to some limit profiles, having disjoint supports: \emph{the segregated states}.
\begin{theorem}[\cite{CoTeVe_AM2005}]\label{thm:intro_locconv}
Let $(u_{1,\beta},\ldots,u_{k,\beta})$, for $\beta>0$, be a family of solutions to system \eqref{eq:generalLV} satisfying a (uniform in $\beta$) $L^\infty_{\loc}(\Omega)$ bound.
Then, up to subsequences,  there exists $(\bar u_1,\dots,\bar u_k)$
such that,
\[
u_{i,\beta}\to \bar u_i\quad\text{ in }H^1_\loc(\Omega)\cap C^{0,\alpha}_\loc(\Omega), \qquad\text{as }\beta\to+\infty,
\]
for every $i=1,\dots,k$ and $0<\alpha<1$. Moreover, the $k$-tuple $(\bar u_1,\dots,\bar u_k)$ is a segregated state:
\[
\bar u_i\bar u_j\equiv 0 \; \text{ in }\Omega.
\]
\end{theorem}
In the last decade, both the asymptotics and the qualitative properties of the limit segregated profiles have been the object of an intensive study, mostly in the symmetric case, by different teams \cite{CoTeVe_CVPDE2005, CoTeVe_IUMJ2005, CoTeVe_IFB2006, MR2529504, MR2863857,MR2595199, MR3375537}. Similarly,  the dynamics of strongly competing species has been addressed as a singularly perturbed parabolic reaction-diffusion system in connection with spatially segregated limit profiles in \cite{MR2846360, MR1900331,MR2890532,MR2776460}.  In the quoted papers, a special attention was paid to the structure of the common zero set of these limit profiles. The following theorem collects the main known facts about the geometry  of the nodal set in the stationary symmetric case:
\begin{theorem}[\cite{CoTeVe_AM2005, MR2529504, MR2483815, MR2984134}]
\label{teo:nodal_set_symmetric}
Assume that
\[
a_{ij}=a_{ji}>0,\qquad\text{for every }i\neq j.
\]
Let $\bar U=(\bar u_1,\dots,\bar u_k)$ be a segregated limit profile
as in Theorem \ref{thm:intro_locconv}, and let
$\zeroset=\{x\in\Omega\;:\;\bar U(x)=0\}$ its nodal set.
Then, there exist complementary subsets $\regset$ and $\singset$  of $\zeroset$,  respectively the regular part, relatively open in $\zeroset$ and the singular part, relatively closed, such that:
\begin{itemize}
\item {$\regset$ is a collection of  hyper-surfaces of class $C^{1,\alpha}$} (for every $0<\alpha<1$),
and for every $x_0\in \regset$
\begin{equation*}\label{eq:reflection_principle}
\lim_{x\rightarrow x_0^+} |\nabla \bar U(x)|=  \lim_{x\rightarrow x_0^-} |\nabla \bar U(x)|\neq 0,
\end{equation*}
where the limits as $x\to x_0^\pm$ are taken from the opposite sides of the hyper-surface;
\item {$\mathcal{H}_\mathrm{dim} (\singset)\leq N-2$}, and if $x_0\in\singset$ then
$\lim_{x\rightarrow x_0} |\nabla \bar U(x)|= 0$.
\end{itemize}
Furthermore,  if $N=2$, then $\zeroset$ consists in a locally finite collection of curves meeting with definite semi-tangents with equal angles at a locally finite number of singular points.
\end{theorem}
On the contrary, the present paper is concerned with \emph{asymmetric inter-specific competition rates}, with the purpose to highlighting the substantial differences with the symmetric case in two space dimensions.  To describe our main result, we consider a simplified, yet prototypical, boundary value problem with  competition terms of Lotka-Volterra type:
\begin{equation}\label{eq:betasys}
 \begin{cases}
  -\Delta u_i = -\beta u_i \sum_{j\neq i} a_{ij} u_j & \text{in }\Omega\\
  u_i = \varphi_i & \text{on }\partial\Omega,
 \end{cases}
 \qquad i=1,\dots,k.
\end{equation}
Throughout the whole paper we will assume that:
\renewcommand\labelenumi{(A\theenumi)}
\begin{enumerate}
	\item\label{ass:1} $\Omega\subset\R^2$ is a simply connected, bounded domain of class $C^{1,\alpha}$;
	\item $a_{ij}>0$ for every $j\neq i$;
	\item $\varphi_i\in C^{0,1}(\partial\Omega)$, $\varphi_i \geq 0$,
	$\varphi_i \cdot \varphi_j \equiv 0$ for every $1\leq i \neq j\leq k$;
	\item\label{eq:ass_non_deg_trace} the trace function $\varphi=\sum_{i=1}^k \varphi_i$ has only non-degenerate
	zero, that is,
	\begin{equation*}
		\forall \bx_0 \in \partial \Omega, \quad \varphi(\bx_0) = 0 \implies \liminf_{\bx \to \bx_0}
		\frac{\varphi(\bx)}{|\bx - \bx_0|} \geq C > 0.
	\end{equation*}
\end{enumerate}
\renewcommand\labelenumi{\theenumi.}
We are interested in component-wise non-negative solutions. As we mentioned, such solutions
segregate as $\beta\to+\infty$. Moreover, exploiting the cancelation properties of system \eqref{eq:betasys}, 
one can prove that a system of differential inequalities passes to the limit as well. More precisely, following
\cite{CoTeVe_AM2005}, let us introduce the following functional class:
\begin{equation}\label{essse*}
\Scal=
 \left\{U=(u_1,\cdots,u_k)\in (H^1(\Omega))^k:\,
  \begin{array}{l}
   u_i\geq0,\, u_i=\varphi_i\mbox{ on }\partial\Omega \\
   u_i\cdot u_j=0\text{ if }i\neq j \\
   -\Delta u_i\leq 0,\,-\Delta \widehat u_i\geq 0
  \end{array}
 \right\},
\end{equation}
where the $i$-th \emph{hat operator} is defined on the generic $i$-th component of  a $k$--tuple as
\begin{equation}\label{cappuccio}
\widehat u_i=u_i-\sum_{j\neq i}\dfrac{a_{ij}}{a_{ji}}u_j,
\end{equation}
and the differential inequalities are understood in variational sense. This is a free boundary problem, where the interfaces $\partial\{u_i > 0\}\cap \partial\{u_j > 0\}$, separate the supports of $u_i$ and $u_j$.
We can reformulate Theorem \ref{thm:intro_locconv} as follows. 
\begin{theorem}[\cite{CoTeVe_AM2005}]\label{thm:intro_oldconv}
For every $\beta>0$ there exists (at least) one vector function $(u_{1,\beta},\ldots,u_{k,\beta})\in (H^1(\Omega))^k$, solution of system \eqref{eq:betasys}.
For every sequence of solutions, there exists (at least) one  $(\bar u_1,\dots,\bar u_k)\in\Scal$ and,
up to a subsequence,
\[
u_{i,\beta_n}\to \bar u_i\quad\text{ in }H^1(\Omega)\cap C^{0,\alpha}(\overline{\Omega}),
\]
for every $i=1,\dots,k$ and $0<\alpha<1$.
\end{theorem}
In order to set up our result about the nodal set in the case of asymmetric interspecific competition rates, we need some more notation. For any $U\in\Scal$, we define the \emph{multiplicity} of a point
$\bx\in\overline{\Omega}$, with respect to $U$, as 
\[
m(\bx)=\sharp\left\{i : \left| \{u_i > 0\} \cap B_r(\bx)\right|>0\text{ for every }r>0\right\}.
\]
Our main purpose is to analyze the structure of the free boundary, i.e. the zero set of a $k$-tuple
$U\in\Scal$:
\[
\zeroset= \{\bx \in {\Omega}: u_i(\bx) =0 \text{ for every } i = 1, \dots, k\}.
\]
Such set naturally splits into the union of the \emph{regular part} $\regset=\zeroset_2:=
\{\bx \in \zeroset : m(\bx) = 2\}$, and of the \emph{singular part}
\[
\singset=\zeroset\setminus\zeroset_2.
\]
We collect in the following lemma some elementary properties about the elements of $\Scal$, which have already been obtained in \cite{CoTeVe_IUMJ2005}.
\begin{lemma}[\cite{CoTeVe_IUMJ2005}]\label{thm:intro_S}
Let $U\in\Scal$. Then:
\begin{enumerate}
\item $U\in C^{0,1}(\overline{\Omega})$;
\item if $m(\bx_0)=1$, then there exist $i$ and $r>0$ such that
\(
\Delta u_i=0
\text{ in }B_r(\bx_0)\cap\Omega
\)
(in particular, $\bx_0\not\in\zeroset$);
\item\label{item:3} if $m(\bx_0)=2$, then there exist $i,j$ and $r>0$ such that
\(
\Delta (a_{ji}u_i-a_{ij}u_j)=0
\text{ in }B_r(\bx_0)\cap\Omega;
\)
\item if $\bx_0\in\singset$ then $\lim_{r\to0}\sup_{B_r(\bx_0)} |\nabla u_i|=0$, for every $i$.
\end{enumerate}
\end{lemma}
Notice that properties 2 and 3 in the above lemma are straight consequences of the definitions of
$\Scal$ and $m$. Now we are ready to state  our main result concerning the properties of the segregation boundary.
\begin{theorem}\label{thm topologia sing}
Let (A\ref{ass:1}--\ref{eq:ass_non_deg_trace}) hold, $U\in\Scal$, and $\zeroset=\zeroset_2\cup\singset$. Then:
\begin{enumerate}
 \item $\zeroset_2$ is relatively open in $\zeroset$, and it consists in the finite union of  analytic curves;
 \item $\singset$ is the union of a finite number of isolated points inside $\Omega$;
 \item  for every $\bx_0\in\partial\Omega$, either $m(\bx_0)=1$ or $m(\bx_0)=2$.
\end{enumerate}
Furthermore, if $\bx_0\in\singset$ then
$m(\bx_0)=h\geq 3$, and there exist an explicit constant $\alpha\in\R$
and an explicit bounded function $A=A(\bx)$ such that
\begin{equation}\label{eq:asympt_expans}
\mathcal{U}(r,\vartheta)=A\,r^{\nu}
\cos\left(\frac{h}{2}\vartheta
-\alpha \log r\right)+o(r^{\nu})\qquad\text{as }r\to 0,
\end{equation}
where $(r,\vartheta)$ denotes a (suitably rotated) system of polar coordinates about
$\bx_0$, $\mathcal{U}$ is a suitable weighted sum of the components $u_i$ meeting at $\bx_0$,
\begin{equation}\label{eq:gamma}
\nu = \frac{h}{2} + \frac{2\alpha^2}{h},\qquad\text{ and }\qquad 0< A_0 \leq A(r,\vartheta) \leq A_1.
\end{equation}
In particular, whenever $\alpha\neq0$, the regular part of the free boundary is described asymptotically by
$h$ equi-distributed logarithmic spirals (locally around $\bx_0$).
\end{theorem}
\begin{remark}\label{rem:alpha3}
The value of $\alpha$ in \eqref{eq:asympt_expans}, \eqref{eq:gamma} is explicit in terms of the coefficients $a_{ij}$, with $i$ and $j$ belonging to the set of
indexes associated to the $h\leq k$ densities which do not identically vanish near $\bx_0$ (see equations \eqref{eq:lmbd}, \eqref{eq:alpha} below). For instance,
when $u_1$, $u_2$ and $u_3$ meet at $\bx_0$, with $m(\bx_0)=3$, then (up to a change of sign)
\[
\alpha=\frac{1}{2\pi}\log\left(\dfrac{a_{12}}{a_{21}}\cdot\dfrac{a_{23}}{a_{32}}\cdot\dfrac{a_{31}}{a_{13}}\right).
\]
Consequently, when $\alpha\neq0$, the vanishing order $\nu$ does not depend only on the number of densities involved (as in the symmetric case), but also on the
competition coefficients; moreover, the vanishing order is not forced to be half-integer, in great contrast with the symmetric case. Finally, the function $\mathcal{U}$
agrees with each density near $\bx_0$, up to a constant multiplicative factor, see equation \eqref{eq:vtilde}.
\end{remark}
\begin{remark}
In case $a_{ij}=a_{ji}$ for every $j\neq i$, then $\alpha=0$, and the spirals reduce to straight lines; in this way we recover the equal-angles-property for
multiple points already obtained in \cite{CoTeVe_IUMJ2005}. On the other hand it is easy to choose the competition coefficients to force $\alpha\neq0$. For instance, in the case
of $k=3$ densities on the ball, one can use the construction suggested by our proof, in order to obtain the existence of an element of $\mathcal{S}$, with a prescribed $\alpha$,
for a set of traces having codimension 2 (see Fig. \ref{fig 3 spir}). In the same spirit, for any given real number $\nu\ge3/2$, one can choose the competing coefficients to obtain elements of
$\mathcal{S}$ having a multiple point with vanishing order $\nu$.
\end{remark}
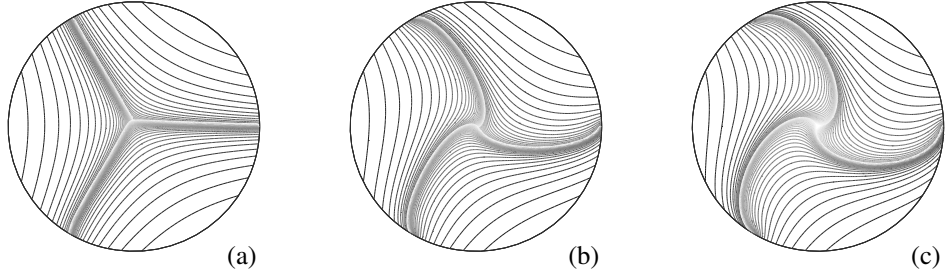
\begin{figure}\label{fig 3 spir}
	\centering
        \begin{tikzpicture}
            \node[anchor=south west,inner sep=0] (image) at (0,0) {\includegraphics[width= .33\textwidth]{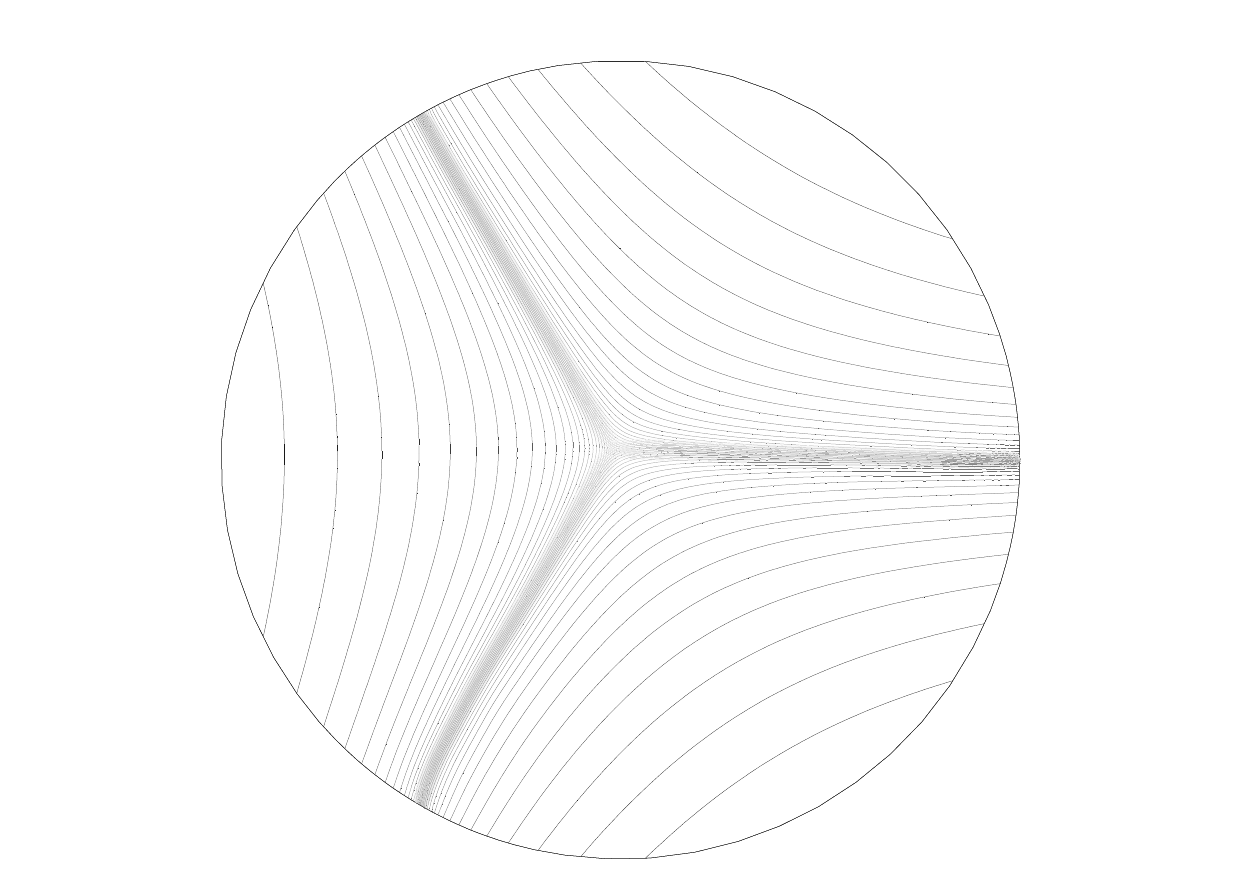}};
            \node[anchor=south west,inner sep=0] (image) at (4,0) {\includegraphics[width= .33\textwidth]{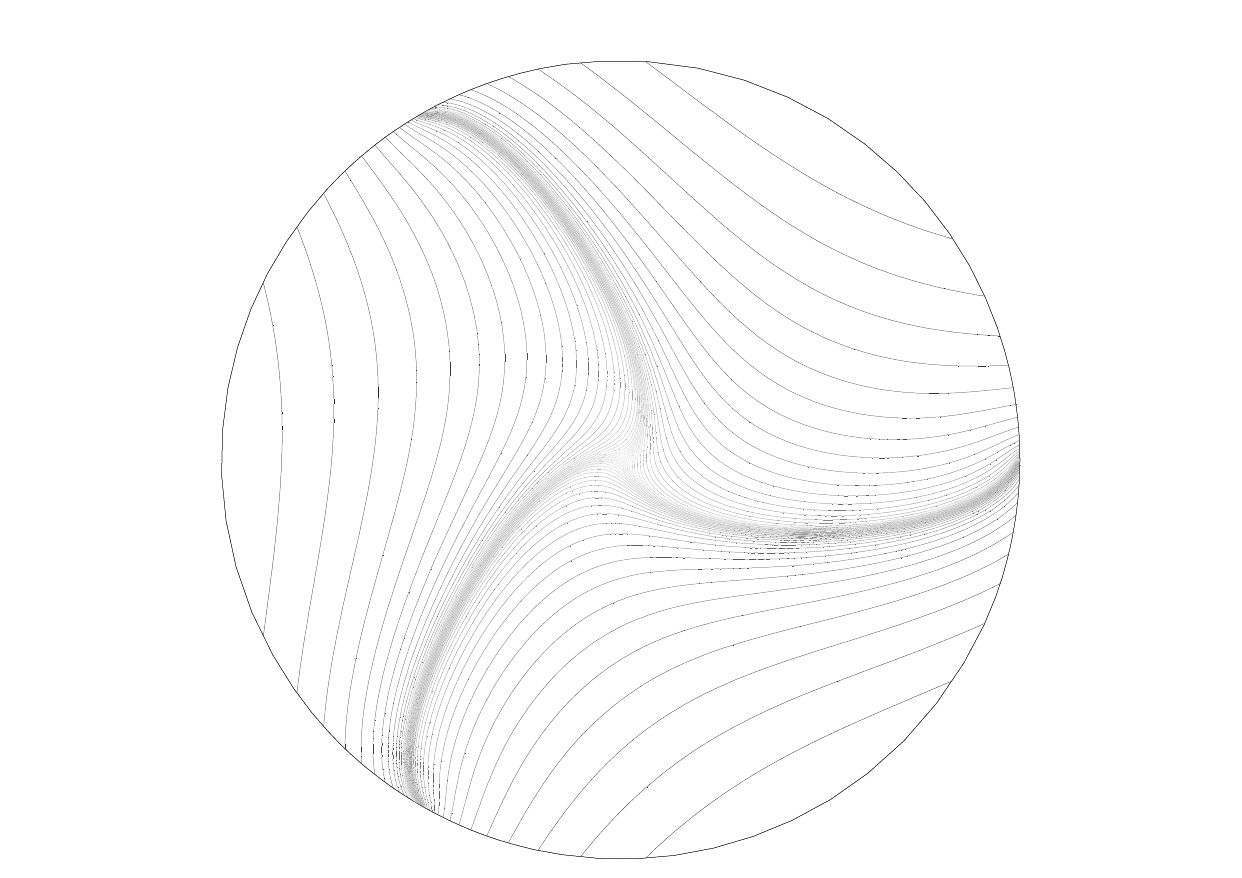}};
            \node[anchor=south west,inner sep=0] (image) at (8,0) {\includegraphics[width= .33\textwidth]{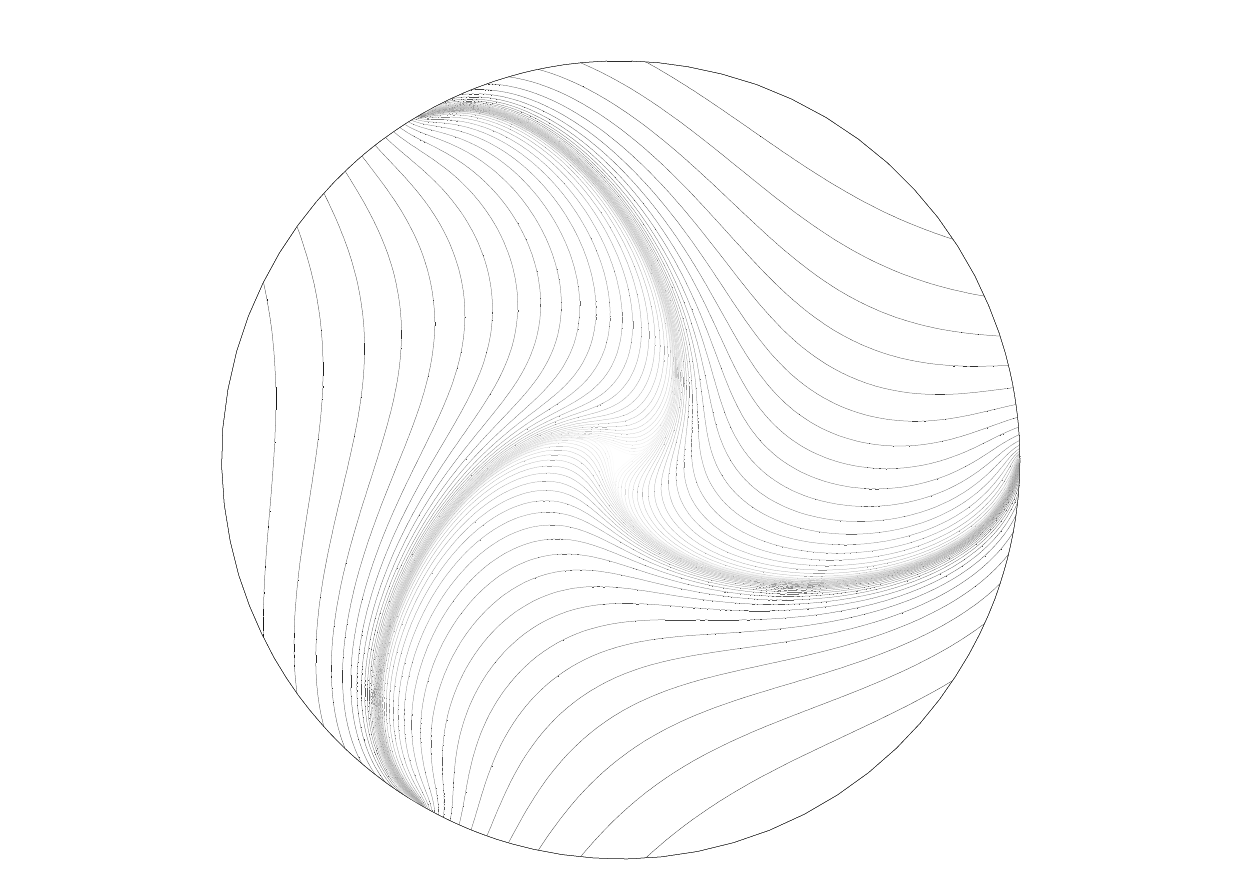}};
            \node[align=center] at (3.5,0) {(a)};
            \node[align=center] at (7.5,0) {(b)};
            \node[align=center] at (11.5,0) {(c)};
        \end{tikzpicture}
\caption{Numerical simulations of functions belonging to the class $\mathcal{S}$ for different values of $\alpha$. In this particular case, we have considered a system of 3 components (labeled in counterclockwise order as $u_1$, $u_2$ and $u_3$) in the unit ball, with boundary conditions given by suitable restrictions of $|\cos(3/2 \, \vartheta)|$. In picture (a), $a_{ij} = 1$ for all $i,j$, which yields $\alpha = 0$ (see Remark \ref{rem:alpha3}). In picture (b), $a_{ij} = 4$ if $j-i = 1 \mod 3$ and $a_{ij} = 1$ otherwise, which yields $\alpha = 3 \log 4 / 2\pi$ ($> 0$, which implies by equation \eqref{eq:asympt_expans} that the free-boundary is described asymptotically by rotations of the clockwise logarithmic spiral $\vartheta =  \log 4/ \pi \, \log r$). In picture (c), $a_{ij} = 10$ if $j-i = 1 \mod 3$ and $a_{ij} = 1$ otherwise, which yields $\alpha = 3\log 10/2 \pi$.}
\end{figure}
\begin{remark}
In the asymmetric case, the nodal partition determined by the supports of the components can neither 
be optimal with respect to any Lagrangian energy, nor be stationary with respect to any domain variation.
Indeed, it is known (cfr. \cite{CoTeVe_AIHPAL2002,MR2393430,MR2483815,MR2984134}) that boundaries of optimal partitions share the same properties
of Theorem \ref{teo:nodal_set_symmetric}. Hence they can not exhibit logarithmic spirals. This fact is in striking contrast with the picture  for symmetric inter-specific competition rates: indeed, in such a case, relatively to Theorem \ref{thm:intro_oldconv}, we know that solutions  to \eqref{eq:betasys} are unique, together with their limit profiles in the class $\Scal$
 (see \cite{CoTeVe_IFB2006,MR2595199, MR2863857}). Hence, though system \eqref{eq:betasys}  does not possess a variational nature, it fulfills a  minimization principle in the segregation limit, while this is impossible in the asymmetric setting.
\end{remark}
\begin{remark}
Nevertheless,  even in the asymmetric case, functions in the class $\Scal$ still share with the solutions of variational problems, including harmonic functions, the following fundamental features:
\begin{itemize}
	\item singular points are isolated and have a finite vanishing order;
	\item the possible vanishing orders are quantized;
	\item the regular part is smooth.
\end{itemize}
It is natural to wonder whether similar analogies still hold dimensions higher than two. As already
remarked, however, new strategies and unconventional techniques have to be designed to treat the
asymmetric case, since a number of standard tools in free boundary problems have to fail in such
situation: for instance, as the planar case shows, the Almgren monotonicity formula can not hold, and
the singularities do not admit, in general, homogeneous blow-ups.
\end{remark}
The proof of Theorem \ref{thm topologia sing} proceeds as follows.  In Section \ref{sec:reduction}, exploiting some topological properties of the zero set of harmonic functions, we show that it suffices to consider the case in which $\Omega = B$, and a unique connected component of $\singset \cap B$ is joined to the boundary by a finite number of smooth curves that describe $\zeroset_2 \cap B$. In Section \ref{sec rep}, assuming that such connected component is given by a point, we give a description of the set $\zeroset_2 \cap B$. To do this, by a conformal mapping, we translate the original problem in the ball $B$ to that of describing the zero set of a harmonic function defined on the half plane. Finally, in Section \ref{sec high} we prove that any connected component of $\singset$ is actually just a point. We achieve this by noticing that in any ball contained in $B$ and whose boundary intersects
$\singset$, the set $\zeroset$ coincides necessarily with that of a harmonic function.

\subsubsection*{Notation} Unless otherwise specified, we adopt the following conventions:
\begin{itemize}
	\item points in $\Omega$ are denoted with $\bx,\by$, and so on; points
	in $\R^2_+$ have coordinates $(x,y)$;
	\item the null set of a $k$-tuple is $\zeroset= \{\bx\in\Omega : u_i(\bx) =0 \; \forall i = 1, \dots, k\}$;
	\item the sets $\omega_i = \{\bx\in\Omega:u_i(\bx) > 0 \}$ (open), $\supp(u_i) =\overline{\omega}_{i}$ (closed);
	\item the multiplicity of a point is defined as the natural number \[m(\bx)=\sharp\left\{i : \left|\omega_i\cap B_r(\bx)\right|>0\text{ for every }r>0\right\};\]
	\item $\zeroset_h = \{\bx \in \zeroset : m(\bx) = h\}$, for $h = 0, 1, \dots, k$;
	\item the singular set is
	$\singset = \zeroset \setminus \zeroset_2 = \zeroset_0 \cup \zeroset_3 \cup \dots \cup \zeroset_k $
	(we will see that no zero of multiplicity $1$ is allowed inside $\Omega$);
	\item $\Gamma_{ij} = \overline{\omega}_{i} \cap \overline{\omega}_{j} \cap \zeroset_2$;
	\item an open curve, or simply a curve when no confusion may arise,
	is a $1$-dimensional manifold (without boundary), i.e. the image of a
	(open) interval through a regular map; in particular, it is locally diffeomorphic to an interval,
	but it may not be rectifiable. In particular, we will show that $\Gamma_{ij}$ is such a curve, as
	far as it is non-empty.
\end{itemize}

\section{Preliminary reduction}\label{sec:reduction}

In this section we show that, without loss of generality, we can reduce to the
model case scenario described in the following assumption (MCS).

\noindent \textbf{Assumption} (MCS). 
Without loss of generality, beyond (A\ref{ass:1}--\ref{eq:ass_non_deg_trace}), we can assume that:
\begin{itemize}
\item $k\geq3$;
\item for each $i = 1, \dots, k$ both $\omega_i$ and $\overline{\omega}_i\cap\partial\Omega$ are connected, simply connected sets;
\item the traces $\varphi_i$ are labelled in counterclockwise sense;
\item $\Gamma_{ij}$ is a non-empty connected open curve whenever $i-j = \pm 1 \mod{k}$, and it is empty otherwise;
\item $\singset$ has a unique connected component, which lies away from $\partial\Omega$
\item $\Omega$ is the unit ball $B=B_1(\bm{0})$, and, for some $|\be|=1$,
\begin{equation}\label{eq:Winthehalfball}
\bm{0}\in \singset \subset \{\bx\in B:\bx\cdot \be\geq0\}.
\end{equation}
\end{itemize}
%
More precisely, we will show the following result.
\begin{proposition}\label{prop:reduction}
$\Omega$ can be decomposed in the finite union of some domains, each of which satisfy (MCS), up to a relabelling of the restricted densities and to some conformal deformations.
\end{proposition}
The rest of this section is devoted to the proof of the above proposition.
We refer the reader also to \cite[Section 7]{CoTeVe_JFA2003}, where a similar preliminary analysis was conducted under the assumption that $a_{ij} = 1$ for every
$i$ and $j$.
\begin{remark}
In order to reduce to assumption (MCS), we will perform a number of operations like dividing
$\Omega$ into subsets, adding and/or relabeling densities, and so on. With some abuse of
notation, we will always write $\Omega$ for the domain and $k$ for the number of densities. Notice
that, once the proposition is proved, the proof of Theorem \ref{thm topologia sing}
will be reduced to show that, under (MCS), $\singset$ consists in a single point, around which the
asymptotic expansion \eqref{eq:asympt_expans} holds true.
\end{remark}
%
%
As a first step we use the maximum principle to reduce to connected, simply connected sub-domains.
\begin{lemma}\label{lem simply conn}
Each $u_i$ is positive and harmonic in $\omega_i$, and there are no interior $1$-multiplicity zeroes:
$\zeroset_1 \cap\Omega = \emptyset$.
Furthermore, possibly by introducing a new family of densities, we have that each $\omega_i$ is connected and simply connected, and
\[
\omega_i\cup
{\{x:\varphi_i(x)>0\}}\text{ is pathwise connected}.
\]
\end{lemma}
\begin{proof}
The first part of the statement follows by Lemma \ref{thm:intro_S}, and from the strong maximum principle. Let us assume that there exists $\bx_0 \in \zeroset_1\cap\Omega$, that is, let us assume that there exists
$\bx_0\in\Omega$, $r > 0$ and $i$ such that $u_i(\bx_0)=0$ and $|\omega_i \cap B_r(\bx_0)| > 0$, while $|\omega_j \cap B_r(\bx_0)| = 0$ for $j \neq i$. But then $u_i$ is harmonic in $B_{r}(\bx_0)$ and, since its boundary data on $\partial B_{r}$ is non-trivial, by the maximum principle we have that $u_i(\bx_0) > 0$, a contradiction. Finally, since $\partial\omega_i\subset \zeroset\cup\partial\Omega$, again the maximum principle implies that each connected component $\sigma$ of some $\omega_i$ must satisfy $\partial\sigma\cap\partial\Omega\neq\emptyset$. By assumption (A\ref{eq:ass_non_deg_trace}) we know that $\partial\Omega\setminus\{\varphi=0\}$ has
a finite number of connected components; we deduce that the same holds also for each $\omega_i$. Introducing, if necessary, further formal densities, we can then assume that each set $\omega_i$ is connected.
Since (by continuity of the densities $u_i$) these sets are open, they are also path-connected. Next, it is easily shown that each of these components is simply connected. Indeed, let $\gamma\subset\omega_i$
be any Jordan curve, and let $\Sigma$ denote the bounded region of $\R^2$ such that $\partial \Sigma = \gamma$. By the maximum principle, all the other components $u_j$, $j\neq i$, when restricted to $\Sigma$, are trivial. But then the function $u_i$ is harmonic in the mentioned set, and again the maximum principle forces $u_i > 0$ in the interior of this set, implying that the curve $\gamma$ is contractible in $\omega_i$. Finally, the pathwise
connectedness of $\omega_i\cup\{x:\varphi_i(x)>0\}$ easily follows from the fact that $\partial\Omega$ is of class $C^{1,1}$.
\end{proof}
\begin{remark}
In principle it may happen that some point $\bx_0\in\partial\Omega$ is a zero of multiplicity $1$. Nonetheless, in such case, there exists an index $i$ for which $\varphi_i$ is positive on both sides of $\bx_0$, and such point is separated from $\zeroset$.
\end{remark}
\begin{lemma}\label{lem:ZconnessoapartialO}
The set $\zeroset \cup \partial \Omega$ is connected.
\end{lemma}
\begin{proof}
Let $O_1, O_2 \subset \R^2$ be two open sets such that $O_1 \cap O_2 = \emptyset$, $(\zeroset \cup \partial \Omega) \subset O_1 \cup O_2$.
Since $\partial\Omega$ is connected, we deduce that one of the sets, say $O_1$, is a subset of $\Omega$ and the other, $O_2$, contains its boundary: $O_1 \subset \Omega$ and $\partial \Omega \subset O_2$. Moreover $\partial O_1 \cap (\zeroset \cup \partial \Omega) = \emptyset$ by definition. We deduce
that each connected component of $\partial O_1$ must be a subset of some $\omega_i$, and by simple connectedness we find that necessarily $O_1 \cap (\zeroset \cup \partial \Omega) = \emptyset$.
\end{proof}
Next we turn to analyze the regular part of the segregation boundary.
\begin{lemma}\label{lem z2 curves}
Each $\Gamma_{ij}$ is either empty or a $C^1$ connected, open curve. In the latter case, $\omega_i \cup \Gamma_{ij} \cup \omega_j$ is an open and simply connected subset of $\Omega$. In particular,
$\zeroset_2 = \cup_{i\neq j} \Gamma_{ij}$ is the disjoint union of a finite number of regular curves.
\end{lemma}
\begin{proof}
We first show that if $\bx_0 \in \zeroset_2$, then $\bx_0 \in \Gamma_{ij}$, which is locally defined by a smooth curve near $\bx_0$. Indeed, let $\bx_0 \in \zeroset_2$ and, by definition, let $i \neq j$ and $r > 0$ be such that $\omega_h \cap B_r(\bx_0)$ is not empty if and only if $h = i,j$. It follows that the function $a_{ji}u_i-a_{ij}u_j$, restricted to $B_r(\bx_0)$, is harmonic and vanishes in $\bx_0$: since $\omega_i$ and $\omega_j$ are path-connected, $\bx_0$ is a simple zero, and the implicit function theorem implies that the set $\Gamma_{ij}$ is represented by a smooth curve, locally near
$\bx_0 \in \Gamma_{ij}$.

We now show the connectedness of $\Gamma_{ij}$: let us consider $\bx_0 \neq \bx_ 1 \in \Gamma_{ij}$. Locally at both $\bx_0$ and $\bx_1$, $\Gamma_{ij}$ is a smooth curve, and since $\omega_i$ and $\omega_j$ are open and (path-)connected, we can easily construct two non self-intersecting curves $\gamma_i:[0,1] \to \supp(u_i)$ and $\gamma_j:[0,1] \to \supp(u_j)$ such that $\gamma_i(0,1) \subset \omega_i$ and $\gamma_j(0,1) \subset \omega_j$ and, moreover, $\gamma_i(0) = \gamma_j(1)= \bx_0$ and $\gamma_i(1) = \gamma_j(0)= \bx_1$. It follows by construction that on the Jordan's curve $\gamma := \{\bx_0\} \cup \gamma_i \cup \{\bx_1\} \cup \gamma_j $ only $u_i$ and $u_j$ are non trivial. Calling $\Sigma$ the open, simply connected region enclosed by $\gamma$, then, by the maximum principle, all the other functions are trivial when restricted to $\Sigma$. It follows in particular that $\Delta (a_{ji} u_i - a_{ij} u_j) = 0$ in $\Sigma$, and it vanishes on $\partial\Sigma$ exactly at $\bx_0$ and $\bx_1$.
Standard properties of harmonic functions imply that $\Gamma_{ij}\cap\Sigma$ is connected.

We are left to show that if $\Gamma_{ij} \neq \emptyset$, then $\omega_i \cup \Gamma_{ij} \cup \omega_j$ is an open and simply connected subset of $\Omega$, but this is an immediate consequence of the above construction.
\end{proof}
The Hopf lemma implies that, for any $i$, some $\Gamma_{ij}$ must be non-empty.
\begin{lemma}\label{lem:Hopf}
Let $B_r(\bx_0)\subset\omega_i$ and $\bp\in\partial B_r(\bx_0)\cap \partial\omega_i$. Then
$m(\bp)=2$.
\end{lemma}
\begin{proof}
Hopf's lemma forces $\nabla u_i(\bp) \neq 0$, and Lemma \ref{thm:intro_S} yields
$\bp \in \zeroset_2 $.
\end{proof}
Turning to the analysis of higher multiplicity points, we first show that the non-degeneracy assumption (A\ref{eq:ass_non_deg_trace}) implies that the singular
set $\singset$ lies in the interior of $\Omega$. We need a preliminary result.
\begin{lemma}\label{lem:halfdisk}
Let $\bx_0\in\partial\Omega$ be such that, for some $i\neq j$, $\bx_0\in\overline{\{x:\varphi_i(x)>0\}}\cap\overline{\{x:\varphi_j(x)>0\}}$,
and let $\bnu$ denote the exterior normal unit vector to $\partial\Omega$ at $\bx_0$. Then:
\begin{enumerate}
 \item  there exist a unit vector $\be$, with $-1/2<\be\cdot\bnu < 0$, and positive constants $\alpha$, $L$,
 such that $\alpha t \leq u_i(\bx_0 + t \be) \leq Lt$, for $t>0$ sufficiently small;
 \item there exist constants $M>0$, $\gamma>1$ such that $|u_h(\bx)|\leq M|\bx-\bx_0|^\gamma$ for every $h\neq i,j$;
 \item if $\Gamma_{ij}=\emptyset$ then there exists $\rho>0$ sufficiently small such that $(B_\rho(\bx_0 + \rho \be) \cap \{\bx:(\bx-\bx_0)\cdot \be^\perp>0\})\subset\omega_i$, where
 the unit vector $\be^\perp$, orthogonal to $\be$, is chosen such that $\be^\perp\cdot\bnu < 0$ (see Fig. \ref{fig:halfdisk}).
\end{enumerate}
\end{lemma}
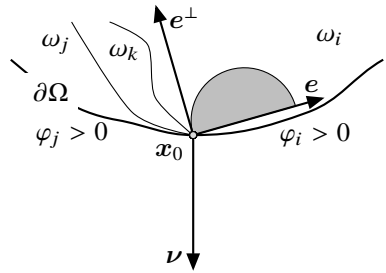
\begin{figure}[ht]
\begin{center}
\begin{tikzpicture}
\draw[thick] (-2.4,1) to[out=-40,in=170] node[pos=.4,fill=white] {$\partial\Omega$} (-1,.2) to[out=-10,in=180] (0,0) to[out=0,in=200] (1.5,.3) to[out=20,in=210] (2.5,1);
\draw[draw=black,fill=gray!50] (0,0) arc (196:16:.7) -- cycle;
\draw[thick,->] (0,0) -- node[pos=.9,left] {$\bnu$} (-90:1.8);
\draw[thick,->] (0,0) -- node[pos=.9,above] {$\be$} (16:1.8);
\draw[thick,->] (0,0) -- node[pos=.9,right] {$\be^\perp$} (106:1.8);
\draw plot[smooth]
  coordinates{
   (-1.6,1.5) (-.8,.3) (0,0)
  };
\draw plot[smooth]
  coordinates{
   (-1.1,1.5) (-.6,1.2) (-.5,.5) (0,0)
  };
\draw[mark options={fill=gray!30!white}, only marks, mark size=1.3, mark=*]
plot (0,0) node[below left] {$\bx_0$};
\draw (1.8,1.3) node {$\omega_i$};
\draw (-1.8,1.2) node {$\omega_j$};
\draw (-0.9,1.1) node {$\omega_k$};
\draw (1.6,-.0) node {$\varphi_i>0$};
\draw (-1.6,-.0) node {$\varphi_j>0$};
\end{tikzpicture}
\caption{The half-disk $B_\rho(\bx_0 + \rho \be) \cap \{\bx:(\bx-\bx_0)\cdot \be^\perp>0\}$,
contained in $\omega_i$ when $\Gamma_{ij}$ is empty and $\rho$ is small (Lemma \ref{lem:halfdisk}).\label{fig:halfdisk}}
\end{center}
\end{figure}
\begin{proof}
We assume w.l.o.g. $\bx_0=(0,0)$. All the following arguments are understood as local (near $0$).

We start by proving 1. Recall that, by Lemma \ref{thm:intro_S}, $U$ is Lipschitz with constant, say, $L>0$. Since $U(\bx_0)=0$, this immediately yields $u_h(\bx)\leq L |\bx-\bx_0|$,
for every $h$ and $\bx$. On the other hand, since $\partial\Omega$ is of class $C^{1,1}$, we can assume that $\Omega=\{(x,y)\in\R^2 : y>f(x)\}$, with $f\in C^1$ and
$f(0)=f'(0)=0$. Assumption (A\ref{eq:ass_non_deg_trace}) implies, for instance, $\varphi_i(x,f(x))>C x$ (resp. $\varphi_j(x,f(x))>-C x$) for $x>0$ (resp. $x<0$). Then, for any $m=\tan\theta>0$ sufficiently
small, there exists $\alpha>0$ such that
\begin{equation}\label{eq:nondegtr}
u_i(x,mx) \geq u_i(x,f(x))  - L |mx - f(x)| \geq (C-Lm)x + o(x) \geq \alpha x,
\end{equation}
for $x>0$ small, and 1 follows, with $\be = (\cos\theta,\sin\theta)$. Note that, at least for a smaller $m$, a similar estimate holds true also for $u_j$:
\[
u_j(x,-mx) \geq  -\alpha x,\qquad \text{for $x<0$ small}.
\]

From the previous point we have that (locally near $0$) $\omega_h$ is contained in the angle
$\{(x,y):y\geq m|x|\}$ whenever $h\neq i,j$. Standard comparison arguments with the positive harmonic function of a cone yield point 2, with $C= \|U\|_{L^\infty}$ and $\gamma = \pi /(\pi - 2\theta)$.

Finally, let $\rho>0$ small to be fixed. Notice that, choosing $\delta = \alpha\cos\theta/(2L)$, we have
\begin{equation}\label{eq:barrier1}
u_i|_{\partial B_{\delta \rho}(\rho \be)\cap\Omega}\geq u_i(\rho\be) - L\delta\rho\geq \frac{\alpha}{2}\rho\cos\theta
\end{equation}
(in particular, $0<\delta<1$). Let us consider the number
\[
\bar r := \sup\{ r>0 : (B_r(\rho\be) \cap \{(x,y):y>mx\})\subset\omega_i \}.
\]
On the one hand, by \eqref{eq:barrier1}, we have that $\bar r > \delta\rho$. On the other hand, since $\bx_0=0\not\in\omega_i$,
it holds $\bar r\leq \rho$. Since 3 is equivalent to $\bar r = \rho$, to conclude we assume by contradiction that $\bar r < \rho$.
This implies the existence of an index $h\neq i$ and of a point $\bp\in\overline{\omega}_i\cap\overline {\omega}_h$, with $\bp\in\Omega$ and $|\rho\be -\bp|=\bar r$.
We have that $B_{\bar r}(\rho\be)$ is both an interior (half-)ball touching at $\bp$ for $\omega_i$,
and an exterior one for $\omega_h$. In particular, by Lemma \ref{lem:Hopf} we infer that $m(\bp)=2$;
then Lemma \ref{thm:intro_S} implies
\begin{equation}\label{eq:match_grad}
\lim_{\substack{\bx\to \bx_0\\ \bx\in\{u_i>0\}}}\nabla u_i(x)=
-\frac{a_{ih}}{a_{hi}}\lim_{\substack{\bx\to \bx_0\\ \bx\in\{u_h>0\}}}\nabla u_h(x),
\end{equation}
where, since $\Gamma_{ij}=\emptyset$, $h\neq j$.

Now, by construction
\[
A_1 := ((B_{\bar r}(\rho\be)\setminus B_{\delta\rho}(\rho\be))\cap \{(x,y):y>mx\})\subset\omega_i;
\]
recalling \eqref{eq:nondegtr} and \eqref{eq:barrier1}, the function
\[
\eta_1(\bx) = \frac{\log \bar r - \log|\bx - \rho\be|}{\log \bar r - \log\delta \rho}\,\frac{\alpha}{2}\rho\cos\theta
\qquad\text{satisfies}\qquad
\begin{cases}
-\Delta \eta_1 = 0 & \text{in }A_1\\
\eta_1 \leq u_i & \text{on }\partial A_1,
\end{cases}
\]
so that
\begin{equation}\label{eq:palla_int}
\lim_{\substack{\bx\to \bp\\ \bx\in\omega_i}}|\nabla u_i(x)| \geq |\nabla\eta_1(\bp)|
= \frac{1}{\log \bar r /(\delta\rho)} \cdot \frac{1}{\bar r} \cdot \frac{\alpha}{2}\rho\cos\theta
\geq \frac{\alpha\cos \theta}{-2\log\delta} >0,
\end{equation}
independently of $\rho$. On the other hand, recalling point 2, we can easily construct a barrier from above for $u_h$. Indeed, let
\[
A_2 := (B_{2\bar r}(\rho\be)\setminus B_{\bar r}(\rho\be))\cap\Omega,
\qquad
\eta_2(\bx) = \frac{\log 2\bar r - \log|\bx - \rho\be|}{\log 2}\,M(2\rho)^\gamma;
\]
then
\[
\begin{cases}
-\Delta u_h \leq 0
= -\Delta \eta_2 & \text{in }A_2\\
u_h \leq \eta_2 & \text{on }\partial A_2.
\end{cases}
\]
We deduce that
\[
\lim_{\substack{\bx\to \bp\\ \bx\in\omega_h}}|\nabla u_h(x)| \leq |\nabla\eta_2(\bp)| = \frac{1}{\log 2} \cdot \frac{1}{\bar r} \cdot M(2\rho)^\gamma\leq \frac{M2^\gamma}{\delta\log2}
\rho^{\gamma-1},
\]
which is in contradiction, when $\rho$ is sufficiently small, with \eqref{eq:palla_int}
and \eqref{eq:match_grad}.
\end{proof}
\begin{corollary}\label{coro:noWboundary}
Let $\bx_0\in\overline{\{x:\varphi_i(x)>0\}}\cap\overline{\{x:\varphi_j(x)>0\}}$. Then $\Gamma_{ij}\neq\emptyset$.
\end{corollary}
\begin{proof}
Applying the above lemma twice (the second time exchanging the role of $i$ and $j$), we obtain
the existence of $\be$, $\be'$, $\rho$, $\rho'$ such that $-1<\be\cdot\be'<0$ and 
\[\begin{split}
(B_\rho(\bx_0 + \rho \be) \cap \{\bx:(\bx-\bx_0)\cdot \be^\perp>0\})\subset\omega_i,\\
(B_{\rho'}(\bx_0 + \rho' \be') \cap \{\bx:(\bx-\bx_0)\cdot (\be')^\perp>0\})\subset\omega_j,
\end{split}\]
a contradiction since
\[\begin{split}
(B_\rho(\bx_0 + \rho \be) \cap &\{\bx:(\bx-\bx_0)\cdot \be^\perp>0\})\\ &\cap
(B_{\rho'}(\bx_0 + \rho' \be') \cap \{\bx:(\bx-\bx_0)\cdot (\be')^\perp>0\})\neq
\emptyset.\qedhere
\end{split}\]
\end{proof}
\begin{lemma}\label{lem:noWonpartialOmega}
Let $\bx_0\in\partial\Omega$. Then either $m(\bx_0)=1$ or $m(\bx_0)=2$. In particular,
$\overline{\singset}\subset\Omega$.
\end{lemma}
\begin{proof}
If $\varphi_i(\bx_0)>0$ for some $i$, then $m(\bx_0)=1$. On the other hand, let $\varphi_i(\bx_0)=0$ for every $i$. Note that assumption (A\ref{eq:ass_non_deg_trace}) implies that $\bx_0$ is an isolated
zero of the trace function $\varphi$. We deduce the existence of two points $\bx_\pm\in\partial\Omega$ with the properties that $\varphi$ is strictly positive in $\bx_\pm$
and also on the curve $\wideparen{\bx_-\bx_0}$, $\wideparen{\bx_0\bx_+}$
(here we denote with $\wideparen{\bx_1\bx_2}$  the relatively open portion of $\partial\Omega$ having counterclockwise ordered endpoints $\bx_1$ and
$\bx_2$, respectively). Two cases may occur.

\textbf{Case 1: both $\bx_\pm\in\{\varphi_i>0\}$, for some $i$.} For $r>0$ sufficiently small we have $B_r(\bx_-)\cap\Omega\subset \omega_i$, $B_r(\bx_+)\cap\Omega\subset \omega_i$.
Recalling that $\omega_i$ is pathwise connected (Lemma \ref{lem simply conn}), we can find a path $\gamma$ such that $\gamma(0)=\bx_+$, $\gamma(1)=\bx_-$ and $\gamma(0,1)\subset\omega_i$.
Let us denote with $\Sigma$ the bounded connected component of $\R^2 \setminus\left(\gamma([0,1])\cup\wideparen{\bx_-\bx_+}\right)$. Then we have that the only nontrivial density on $\partial\Sigma$,
and hence on $\Sigma$ is $u_i$, so that $m(\bx_0)=1$.

\textbf{Case 2: $\bx_-\in\{\varphi_i>0\}$, $\bx_+\in\{\varphi_j>0\}$, with $j\neq i$.} Then Corollary \ref{coro:noWboundary} applies, and $\Gamma_{ij}\neq\emptyset$. Consequently, this case can be treated in a similar way than the previous one, with the only difference
that now $\gamma$ can be constructed in such a way that $\gamma(0,1)\subset\omega_i \cup \Gamma_{ij} \cup \omega_j$ (which is pathwise connected by Lemma \ref{lem z2 curves}). Then $u_h$ vanishes in $\overline{\Sigma}$, for every $h\neq i,j$, and $m(\bx_0)=2$.
\end{proof}
Notice that, at this level, we can not exclude that $\zeroset_0$ is not empty (actually, this will be ruled out in Section \ref{sec high}). Anyhow, by definition, such set coincides with the interior of $\singset$, 
while $\partial\singset$ consists of the points of higher multiplicity. We analyze these points in the next two lemmas.
\begin{lemma}\label{lem bound w and z2}
The boundary of $\singset$ is the accumulation set of $\zeroset_2$:
\[
	\partial \singset \subset \bigcup_{i \neq j} \overline{\Gamma}_{ij}.
\]
\end{lemma}
\begin{proof}
Let $\bx_0 \in \partial \singset$, we need to show that for all $r > 0$ there exists $\bp \in B_r(\bx_0) \cap \zeroset_2$. Fixing $r>0$, since $\bx_0 \in \partial W$ there exists $\bx' \in B_r(\bx_0) \setminus \singset$. Then either $\bx' \in \zeroset_2$ and we are done, or $u_i(\bx') > 0$ for some index $i$. In the latter case, let $R > 0$ be such that $B_R(\bx') \subset \{u_i > 0\} \cap B_r(\bx_0)$ and let us consider the segment $t \in [0,1] \mapsto (1-t)\bx' + t \bx_0$. Since $ \{u_i > 0\} \cap B_r(\bx_0)$ is an open set, there exists a maximum value $\bar t \in [0,1]$ such that
\[
	B_R( (1-t)\bx' + t \bx ) \subset \{u_i > 0\} \qquad \text{for all $t \in [0,\bar t]$}.
\]
Consequently $B_R( (1-\bar t)\bx' + \bar t \bx )$ satisfies the assumptions of Lemma \ref{lem:Hopf}.
\end{proof}
\begin{lemma}\label{lem connect omega lim}
For any non empty $\Gamma_{ij}$, each of its limit sets is either a point of $\partial\Omega$ or a connected subset of $\singset$.
\end{lemma}
\begin{proof}
Being $\Gamma_{ij}$ a locally smooth (non self-intersecting) curve contained in $\Omega$, we immediately obtain that it admits a global one-to-one parametrization $\phi \in C^1(I; \Omega)$, for
some interval $I\subset\R$: for instance, we can take $\phi$ as a solution to the system
\[
	\begin{cases}
		\dot\phi(t) = J \nabla (a_{ji} u_i- a_{ij} u_j)(\phi(t))	\\
		\phi(0) = \bx_0
	\end{cases}
\]
where $J$ is the symplectic matrix and $\bx_0 \in \Gamma_{ij}$ is any point. Since in the set
$\omega_i \cup \Gamma_{ij} \cup \omega_j$ the function $a_{ji} u_i- a_{ij} u_j$ is locally smooth
and has non vanishing gradient, we obtain the existence of $-\infty\leq a < 0 < b \leq + \infty$,
the maximal times of definition of $\phi$, and of the $\alpha$- and $\omega$-limit sets of $\phi$:
\[
	\alpha(\phi) = \bigcap_{a < t < 0} \overline{\phi((a,t))}, \qquad \omega(\phi) = \bigcap_{b > t > 0} \overline{\phi((t,b))}.
\]
It is easy to check that both limits are non-empty, closed, connected subsets of
$\overline{\Omega}$. We consider the set $\omega(\phi)$, the other is analogous. Let
$\bx_0\in\omega(\phi)$. By construction we have that $m(\bx_0)\ge2$. Hence, either
$\omega(\phi)\subset\singset$, or $m(\bx_0)=2$. In the latter case, we see that
$\bx_0\not\in\Omega$, otherwise we may solve the above Cauchy problem in $B_r(\bx_0)$,
contradicting the maximality of $b$. Therefore $\bx_0\in\partial\Omega$, and Lemma
\ref{lem:noWonpartialOmega} forces $\omega(\phi)\equiv\bx_0$.
\end{proof}
Finally, we are in a position to introduce the cut procedures which will yield Proposition \ref{prop:reduction}.
\begin{lemma}\label{lem only one cc}
We can reduce the problem to the case in which, for each $i =1, \dots, k$, $\supp(u_i) \cap \partial \Omega$ is a connected curve. Furthermore, we can assume that any curve $\Gamma_{ij}$ reaches the boundary at most once, there are at least three non trivial densities in $\Omega$ and $\singset$ is not empty.
\end{lemma}
\begin{figure}[th]
\begin{center}
\begin{tikzpicture}
\draw[thick] (0,0) to[out=180,in=-180+165] +(-1,1) node[below left] {$\Omega$} to[out=165,in=-90] +(-.5,1) to[out=90,in=180]
+(1.5,1);
\draw[thick] (0,0) node {$|$} -- node[sloped,pos=.5,below] {$\varphi_1>0$} (2,0) node {$|$};
\draw[thick] (0,3) node {$|$} -- node[sloped,pos=.5,above] {$\varphi_1>0$} (2,3) node {$|$};
\draw[thick] (2,3) to[out=0,in=100] +(1.5,-1.5) to[out=-80,in=0] +(-1.5,-1.5);
\draw plot[smooth]
  coordinates{
   (0,0) (-.1,.6) (-.4,1.2) (-0.5,2) (0,3)
  };
\draw plot[smooth]
  coordinates{
   (2,0) (1.8,1.3) (2,3)
  };
\draw[dashed]
(1,0) -- node[pos=.6,fill=white] {$\gamma$} (1,3);
\draw[mark options={fill=gray!30!white}, only marks, mark size=1.3, mark=*]
plot (1,0) node[above left] {$\bx_0$} plot (1,3) node[below right] {$\bx_1$} ;
\draw (0.4,1) node {$\omega_1$};
\end{tikzpicture}
\qquad
\begin{tikzpicture}
\draw[thick] (0,0) to[out=180,in=-180+165] +(-1,1) node[below left] {$\Omega'$} to[out=165,in=-90] +(-.5,1)  to[out=90,in=180]
+(1.5,1);
\draw[thick] (0,0) node {$|$} -- node[sloped,pos=.7,below] {$\varphi_1'>0$} (.8,0)
arc (-90:0:.2) -- (1,2.8) arc (0:90:.2) -- (0,3) node {$|$};
\draw[thick] (3,0) node {$|$} -- (2.2,0) arc (270:180:.2) -- (2,2.8) arc (180:90:.2)
 -- node[sloped,pos=.3,above] {$\varphi_1''>0$} (3,3) node {$|$};
\draw[thick] (3,3) to[out=0,in=100] +(1.5,-1.5) to[out=-80,in=0] +(-1.5,-1.5);
\draw plot[smooth]
  coordinates{
   (0,0) (-.1,.6) (-.4,1.2) (-0.5,2) (0,3)
  };
\draw plot[smooth]
  coordinates{
   (3,0) (2.8,1.3) (3,3)
  };
\draw (0.4,1) node {$\omega_1'$};
\draw (2.5,2) node {$\omega_1''$};
\draw (4.3,2.8) node {$\Omega''$};
\end{tikzpicture}
\caption{Splitting for Lemma \ref{lem only one cc}: on the left, the original domain; on the right, the split components.\label{fig:split1}}
\end{center}
\end{figure}
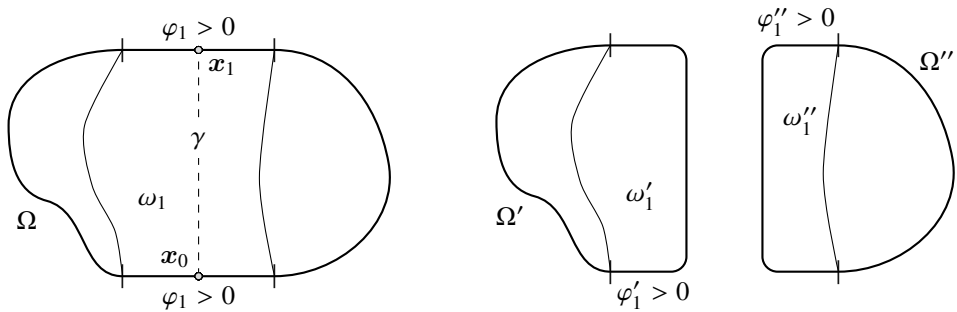
\begin{proof}
By assumption (A\ref{eq:ass_non_deg_trace}) and Lemma \ref{lem:noWonpartialOmega}, $\supp(u_i) \cap
\partial \Omega$ has a finite number of connected components, on each of which $u_i$ does not
identically vanish. If, say, $\supp(u_1) \cap \partial \Omega$ contains more than one connected
component, let $\bx_0$ and $\bx_1$ be any two points that belong to different connected components,
with $u_1(x_i)>0$. Let $\gamma:[0,1]\to\supp(u_1)$ be a smooth simple curve such that
$\gamma(0,1) \subset \omega_1$, $\gamma(0) = \bx_0$ and $\gamma(1) = \bx_1$: $\gamma$ cuts the
domain $\Omega$ in two subdomains, on which the number of connected components of $\supp(u_1) \cap
\partial \Omega$ has reduced by at least one. Moreover, the two subdomains are regular except for
two corner points, in $\bx_1$ and $\bx_2$. Since $u_1(x_i)>0$, one can easily cut neighborhoods of
$\bx_i$ in such a way that the new subdomains are smooth, and all the assumptions hold, in
particular assumption (A\ref{eq:ass_non_deg_trace}) because $U$ is Lipschitz (see Fig. \ref{fig:split1}).
Iterating the previous construction a finite number of times we can assume that each $\supp(u_i)
\cap \partial\Omega$ is a connected curve.

Next, let us assume that $\overline{\Gamma}_{{12}} \cap \partial \Omega = \{\by_0, \by_1\}$.
Note that $\by_0\neq \by_1$. By construction, each of the two connected components of $\partial
\Omega \setminus \{\by_0, \by_1\}$ must coincide with $\supp(u_i) \cap \partial \Omega$ for either $i=1$ or $i=2$. We deduce that $k=2$, $\singset=\emptyset$ and $\zeroset \cap \Omega = \zeroset_2 \cap \Omega = \Gamma_{12} \cap \Omega$, thus the regularity of the free boundary follows from the previous discussion (actually, it is the zero set of the harmonic function $a_{ji} u_i- a_{ij} u_j$). Therefore we are left to deal with the case $k\geq3$, in which each $\Gamma_{ij}$ has at least one limit in $\singset$.
\end{proof}
\begin{lemma}\label{lem:counterclockwise}
Up to a further reduction, we can label the densities $u_i$ in a counterclockwise sense,
in such a way that $\Gamma_{ij} \neq \emptyset$ if and only if $j - i = \pm 1 \mod{k}$, and
both $\overline{\Gamma}_{{ij}} \cap \partial \Omega$ and $\overline{\Gamma}_{{ij}} \cap \singset$
are non-empty.
\end{lemma}
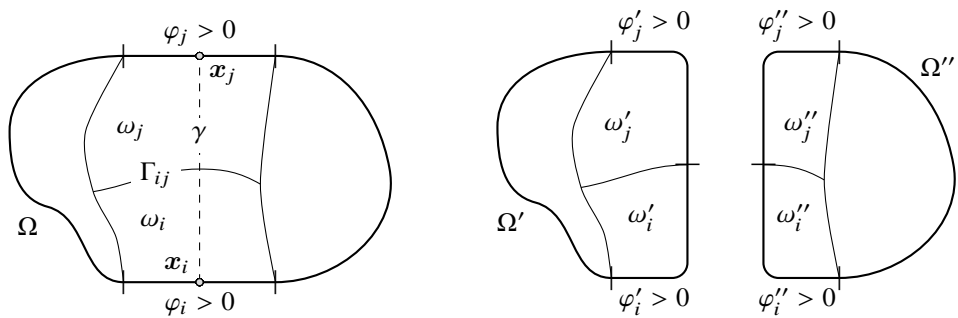
\begin{figure}[th]
\begin{center}
\begin{tikzpicture}
\draw[thick] (0,0) to[out=180,in=-180+165] +(-1,1) node[below left] {$\Omega$} to[out=165,in=-90] +(-.5,1) to[out=90,in=180]
+(1.5,1);
\draw[thick] (0,0) node {$|$} -- node[sloped,pos=.5,below] {$\varphi_i>0$} (2,0) node {$|$};
\draw[thick] (0,3) node {$|$} -- node[sloped,pos=.5,above] {$\varphi_j>0$} (2,3) node {$|$};
\draw[thick] (2,3) to[out=0,in=100] +(1.5,-1.5) to[out=-80,in=0] +(-1.5,-1.5);
\draw plot[smooth]
  coordinates{
   (0,0) (-.1,.6) (-.4,1.2) (-0.5,2) (0,3)
  };
\draw plot[smooth]
  coordinates{
   (2,0) (1.8,1.3) (2,3)
  };
\draw (-.4,1.2) to[out=10,in=180] node[pos=.6,fill=white] {$\Gamma_{ij}$} (1,1.5) to[out=0,in=150] (1.8,1.3);
\draw[dashed]
(1,0) -- node[pos=.65,fill=white] {$\gamma$} (1,3);
\draw[mark options={fill=gray!30!white}, only marks, mark size=1.3, mark=*]
plot (1,0) node[above left] {$\bx_i$} plot (1,3) node[below right] {$\bx_j$} ;
\draw (0.4,.8) node {$\omega_i$};
\draw (0.1,2) node {$\omega_j$};
\end{tikzpicture}
\qquad
\begin{tikzpicture}
\draw[thick] (0,0) to[out=180,in=-180+165] +(-1,1) node[below left] {$\Omega'$} to[out=165,in=-90] +(-.5,1)  to[out=90,in=180]
+(1.5,1);
\draw[thick] (0,0) node {$|$} -- node[sloped,pos=.7,below] {$\varphi_i'>0$} (.8,0)
arc (-90:0:.2) -- node[pos=.5,sloped] {$|$} (1,2.8) arc (0:90:.2) -- node[sloped,pos=.3,above] {$\varphi_j'>0$} (0,3) node {$|$};
\draw[thick] (3,0) node {$|$} -- node[sloped,pos=.7,below] {$\varphi_i''>0$} (2.2,0) arc (270:180:.2) -- node[pos=.5,sloped] {$|$} (2,2.8) arc (180:90:.2)
 -- node[sloped,pos=.3,above] {$\varphi_j''>0$} (3,3) node {$|$};
\draw[thick] (3,3) to[out=0,in=100] +(1.5,-1.5) to[out=-80,in=0] +(-1.5,-1.5);
\draw plot[smooth]
  coordinates{
   (0,0) (-.1,.6) (-.4,1.2) (-0.5,2) (0,3)
  };
\draw plot[smooth]
  coordinates{
   (3,0) (2.8,1.3) (3,3)
  };
\draw (-.4,1.2) to[out=10,in=180] (1,1.5);
\draw (2,1.5) to[out=0,in=150] (2.8,1.3);
\draw (0.4,.8) node {$\omega_i'$};
\draw (2.4,.8) node {$\omega_i''$};
\draw (0.1,2) node {$\omega_j'$};
\draw (2.5,2) node {$\omega_j''$};
\draw (4.3,2.8) node {$\Omega''$};
\end{tikzpicture}
\caption{Splitting for Lemma \ref{lem:counterclockwise}: on the left, the original domain; on the right, the split components.\label{fig:split2}}
\end{center}
\end{figure}
\begin{proof}
Taking into account Lemma \ref{lem only one cc}, the lemma will follow once we show that $\Omega$ can decomposed in such
a way that (in each subdomain) $\Gamma_{ij} \neq \emptyset$ if and only if $\supp{\varphi_i}\cap\supp{\varphi_j}\neq\emptyset$. Let $\Gamma_{ij}\neq\emptyset$ and $\supp{\varphi_i}\cap\supp{\varphi_j}=\emptyset$;
Lemmas \ref{lem simply conn} and \ref{lem z2 curves} imply that $\{\varphi_i>0\}\cap\omega_i\cap\Gamma_{ij}\cap\omega_j\cap\{\varphi_j>0\}$ is pathwise connected. As a consequence, reasoning as in the previous lemma, we
can construct a smooth simple curve $\gamma:[0,1]\to\supp(u_i)\cup\supp(u_j)$ with
$\varphi_i(\gamma(0)) >0$, $\varphi_j(\gamma(1))>0$, and which intersects $\Gamma_{ij}$ just once (and transversally). Again, $\gamma$ cuts the
domain $\Omega$ in two subdomains, on both of which $\supp{u_i}$ intersects $\supp{u_j}$ also on the boundary. As before, the two subdomains are regular except for
two corner points, which can be treated cutting some neighborhoods (see Fig. \ref{fig:split2}). Iterating
the procedure a finite number of times, and relabelling the components, the lemma follows.
\end{proof}
\begin{lemma}\label{ref zero conn}
Under the previous reductions, the set $\zeroset$ is connected.
\end{lemma}
\begin{proof}
Let us assume that $\zeroset$ is not connected. Then there exists a Jordan curve $\gamma$ that separates $\zeroset$ in two components, and in particular $\gamma \cap \zeroset = \emptyset$. Since $\zeroset \cup \partial \Omega$ is connected (Lemma \ref{lem:ZconnessoapartialO}), $\gamma$ must cross $\partial \Omega$ in at least two points, and we can assume that $\gamma$ has a finite (even) number of transverse intersections with  $\partial \Omega$.
Let $t_0 <  t_1 \in [0,1)$ be such that
\[
	\gamma(t_0), \gamma(t_1) \in \partial \Omega,\qquad  \gamma(t) \in \Omega \quad \text{for $t \in (t_0, t_1)$;}
\]
from the discussion above, there exists $i$ such that $\gamma([t_0, t_1]) \in \supp(u_i)$ (otherwise $\gamma$ would intersect $\zeroset$). Let $\sigma$ be the component of $\partial \Omega \setminus \{\gamma(t_0), \gamma(t_1)\}$ that is contained in $\supp(u_i)$ (which exists by Lemma \ref{lem only one cc}): the simple closed curve
$\gamma([t_0, t_1]) \cup \sigma$ does not intersect $\zeroset$, and thus we can deform continuously $\gamma$ into $\gamma([0,t_0]) \cup \sigma \cup \gamma([t_1,1])$ without ever crossing $\zeroset$. Iterating these steps a finite number of times, we end up with a new Jordan curve
$\gamma'$ that is by construction homotopic to $\gamma$ in $\R^2 \setminus \zeroset$ and such that $\gamma'\cap\Omega=\emptyset$. In particular, no point of $\zeroset$ is contained in the unbounded portion of $\R^2 \setminus \gamma'$, a contradiction.
\end{proof}
\begin{lemma}\label{lem:Wconnected}
Under the previous reductions, the set $\singset$ is connected.
\end{lemma}
\begin{proof}
Assume by contradiction that $\singset$ is the disjoint union of the two closed sets $\singset_1$ and $\singset_2$. By Lemmas  \ref{lem:counterclockwise}, \ref{lem connect omega lim} and \ref{lem:noWonpartialOmega}
we have that, for each $i$, either $\overline{\Gamma}_{(i-1)i}\cap  \singset_1 =\emptyset$ or $\overline{\Gamma}_{(i-1)i}\cap  \singset_2 =\emptyset$ (here $\Gamma_{01}$ stands for $\Gamma_{1k}$).
We deduce that $\zeroset$ is the disjoint union of
\[
\singset_1\cup\left\{\bigcup_i\Gamma_{(i-1)i}:\overline{\Gamma}_{(i-1)i}\cap  \singset_2 =\emptyset\right\}
\text{ and }
\singset_2\cup\left\{\bigcup_i\Gamma_{(i-1)i}:\overline{\Gamma}_{(i-1)i}\cap  \singset_1 =\emptyset\right\},
\]
in contradiction with Lemma \ref{ref zero conn}.
\end{proof}
\begin{proof}[Proof of Proposition \ref{prop:reduction}]
Using Lemmas \ref{lem simply conn}, \ref{lem:noWonpartialOmega}, \ref{lem only one cc},
\ref{lem:counterclockwise} and \ref{lem:Wconnected} we obtain that $\Omega$ splits, up to some
residual set of $\Omega\setminus\zeroset$ (recall Figs. \ref{fig:split1}, \ref{fig:split2}) into
the finite union of domains
satisfying all the required properties in (MCS), but the last. By the Riemann Mapping
Theorem we can assume that, up to a conformal change of coordinates, the domain is the unit ball
$B=B_1(\bm{0})\subset\R^2$. With an abuse of notation, we keep writing $u_i$, $\omega_i$,
$\zeroset$, $\singset$, and so on, for the corresponding transformed objects, defined in $B$ instead
of $\Omega$. Since $\Omega$ is of class $C^{1,\alpha}$, we have that the conformal transformation
can be extended to a $C^{1,\alpha}(\overline{B})$ map onto $\overline{\Omega}$ (see, e.g.,
\cite[Thm. 3.6]{Po_1992}). We deduce that assumption (A\ref{eq:ass_non_deg_trace}) holds true also
for the transformed densities. Now, if $\singset\equiv\{\bm{0}\}$ then the proposition follows;
in the other case, let $r>0$ be such that
\[
\singset\subset \overline{B}_r(\bm{0}),\qquad \singset\cap\partial B_r(\bm{0})\ni\bx_0,
\]
where $ r < 1$ by Lemma \ref{lem:noWonpartialOmega}. Using as a second conformal change of
variables the M\"obius transform which sends $\bm{0}$ to $\bx_0$ (and $B$ onto itself) we have
that the pre-image of  $B_r(\bm{0})$ is contained in a ball touching the origin, concluding the proof.
\end{proof}
\begin{remark}\label{rem:asy1}
We stress that, under the conformal changes of variables we introduced in the proof above,
the local expansion of any transformed function near the origin is the same as that of the
original one near $\bx_0$, up to a similarity transformation (this will provide the asymptotic expansion \eqref{eq:asympt_expans}).
\end{remark}

Once we reduced to work on a (finite number of domains) $\Omega$ satisfying (MCS), we conclude this
section by introducing a suitable conformal map, which provides an equivalent formulation of our
problem for a function defined in the half-plane. By construction, we have that $\omega_i
\cap\partial B$ is a connected curve on $\partial B$ for every index $i$, while each non-empty $\Gamma_{ij}$ consists of a smooth curve going from $\partial B$ to $\singset$. Furthermore,
the densities $u_1,\dots,u_k$  are ordered in counterclockwise order around the origin.

Let
\begin{equation}\label{eq:vtilde}
\mathcal{U}(\bx)=
\begin{cases}
u_1(\bx) &\text{if }\bx\in
\supp(u_1),\\
\displaystyle (-1)^{i-1}\left(\prod_{j=2}^{i} \frac{a_{(j-1)
j}}{a_{j(j-1)}}\right) u_i(\bx) &\text{if }\bx\in
\supp(u_i),\,i=2,\dots,k.
\end{cases}
\end{equation}
Using Lemma \ref{thm:intro_S}, we obtain that
\[
\mathcal{U}\in C^1(B\setminus(\Gamma_{1k}\cup\singset)),\qquad \Delta \mathcal{U}=0\text{ in }B\setminus(\Gamma_{1k}\cup\singset).
\]
For concreteness, from now on we assume that $k$ is  even, the odd case following
with minor changes, see Remark \ref{rem:oddk} below. Let
\begin{equation}\label{eq:lmbd}
\lambda:=\dfrac{a_{k1}}{a_{1k}}\cdot\prod_{j=2}^{k}
\frac{a_{(j-1)j}}{a_{j(j-1)}}>0.
\end{equation}
Using again Lemma \ref{thm:intro_S} we infer that, for every $\bx\in\Gamma_{1k}$,
\begin{equation}\label{vi}
\lim_{\substack{\by\to \bx\\ \by\in \omega_k}}\nabla \mathcal{U}(\by)=
\prod_{j=1}^{k} \frac{a_{(j-1)j}}{a_{j(j-1)}}\lim_{\substack{\by\to \bx\\ \by\in \omega_k}}
\nabla u_k(\by)=\lambda \lim_{\substack{\by\to \bx\\ \by\in \omega_1}}\nabla u_1(\by)=\lambda
\lim_{\substack{\by\to \bx\\ \by\in \omega_1}}\nabla \mathcal{U}(\by).
\end{equation}
Now, if $\lambda\neq1$ then $\mathcal{U}$ is not harmonic on
$B\setminus\singset$. On the other hand, for any $\lambda$,
we can construct a harmonic function by lifting $\mathcal{U}$ to the universal covering
of $B\setminus\singset$. More precisely, we consider the conformal map
\begin{equation}\label{eq:defT}
\begin{split}\mathcal{T} &\from \R^2_+  :=\{(x,y)\in\R^2 : y> 0\} \to  B\setminus\{\bm{0}\},\\
 \mathcal{T}&\from (x,y) \mapsto \bx=(e^{-y}\cos x,e^{-y}\sin x).
\end{split}
\end{equation}
Notice that $\mathcal{T}$ is a universal covering. Up to a rotation, let us assume that
the endpoint on $\partial B$ of $\Gamma_{1k}$ is the point of coordinates $(1,0)$.
The Lifting Theorem provides the existence of a
unique regular curve $\Gamma \subset \overline{\R^2_+} $, containing $(0,0)$,
which lifts $\Gamma_{1k}$. Also $\Gamma+(2\pi,0)$ (the
horizontal translated of $\Gamma$) is a regular curve, it
starts from $(2\pi,0)$, and it has empty intersection with $\Gamma$. Noticing that
$B\setminus(\Gamma_{1k}\cup\singset)$ is simply connected (this can be shown as in Lemma \ref{lem z2 curves}),
we can define the set $S$ as the unique lifting of $B\setminus(\Gamma_{1k}\cup\singset)$ such that
\begin{equation}\label{eq:defS}
\partial S\supset\left[\Gamma\cup(\Gamma+(2\pi,0))\cup([0,2\pi]\times\{0\})\right].
\end{equation}
Since $\mathcal{T}$ is conformal, the function
\begin{equation}\label{eq:defvi}
v(x,y):=\mathcal{U}(\mathcal{T}(x,y))
\end{equation}
is harmonic on $S$, and
\[
\int_{S}|\nabla v|^2\,dxdy<+\infty.
\]
Using \eqref{vi}, we can extend $v$ in such a way that
\[
v(x+2\pi,y)=\lambda v(x,y),
\]
so that $v$ is harmonic in the whole $\R^2_+\setminus\mathcal{T}^{-1}(\singset\setminus\{\bm{0}\})$, and continuous
in $\overline{\R^2_+}$. Resuming, $v$ is a solution of the problem
\begin{equation}\label{eq:main_on_the_halfplane}
\begin{cases}
\Delta v =0 & \text{in }\R^2_+\setminus\mathcal{T}^{-1}(\singset\setminus\{\bm{0}\})\\
v\equiv0 & \text{in }\mathcal{T}^{-1}(\singset\setminus\{\bm{0}\})\\
v(x,0)=v_0(x) & x\in\R \\
v(x+2\pi,y)=\lambda v(x,y) \\
\int_{S}|\nabla v|^2\,dxdy<+\infty,
\end{cases}
\end{equation}
where $v_0$ is a suitable combination of the trace functions $\varphi_i$.
\begin{remark}\label{rem:oddk}
Notice that when $k$ is odd one can double the angle, using the map
\[
\mathcal{T}_2 \from (x,y) \mapsto \bx=(e^{-2y}\cos 2x,e^{-2y}\sin 2x).
\]
In such a way, we can reduce to \eqref{eq:main_on_the_halfplane}, with $\lambda^2$ instead of $\lambda$.
\end{remark}
\begin{remark}\label{rem:Tallamenouno}
The (multi-valued) inverse $\mathcal{T}^{-1}$ satisfies
\[
 \mathcal{T}^{-1}(r\cos\vartheta,r\sin\vartheta) = (\vartheta,-\log{r}),
\]
where we use the polar system around $\bm{0}$, writing $\bx=(r\cos\vartheta,r\sin\vartheta)$. Therefore
\begin{equation}\label{eq:Tallamenouno}
\mathcal{U}(\bx) = v(\vartheta,-\log{r})
\end{equation}
(this will provide the asymptotic expansion \eqref{eq:asympt_expans}).
\end{remark}

\section{A representation formula in $\R^2_+$}\label{sec rep}

In this section we deal with problem \eqref{eq:main_on_the_halfplane}, in the particular case in
which $\singset = \{\bm{0}\}$ and the number of densities is even, say $2\np$
(recall Remark \ref{rem:oddk}).
In such a case, we can assume that the curve $\Gamma$ (the lifting of $\Gamma_{1k}$ to the half plane $\R^2_+:=\{(x,y)\in\R^2:y>0\}$) has $C^{1,\alpha}$, regular parametrization
\[
\Gamma:=\{(x(t),y(t)):t\in[0,+\infty)\},
\quad\text{ satisfying }\;
\begin{cases}
t\mapsto(x(t),y(t))\text{ is injective}\\
(x(0),y(0))=(0,0),\\
y(t)>0 \text{ for }t>0,\\
\lim_{t\to+\infty}y(t)=+\infty,\\
\Gamma \cap (\Gamma + (2\pi,0))=\emptyset
\end{cases}
\]
By connectedness, we have that $\Gamma \cap (\Gamma + (2h\pi,0))=\emptyset$ too, as long
as $h\in\Z\setminus\{0\}$. Under this perspective, the set $S\subset\R^2_+$ is a ``strip'',
i.e. the unique unbounded connected open set having boundary
\[
\partial S:=\Gamma\cup(\Gamma+(2\pi,0))\cup([0,2\pi]\times\{0\}).
\]
\begin{definition}
We denote by
\[
S_y:=\{(x,y)\in S\}
\]
the horizontal sections of $S$, having endpoints
\[\begin{split}
\Xmin(y) &:= \inf\{x:(x,y)\in S\}>-\infty,\\
\Xmax(y) &:= \sup\{x:(x,y)\in S\}<+\infty,
\end{split}\]
and diameter
\[
\diam S_y = \Xmax(y) - \Xmin(y),
\]
respectively (see Figure \ref{fig:strip}). Of course, $(\Xmin(y),y)\in\Gamma$ and
$(\Xmax(y),y)\in\Gamma+(2\pi,0)$.
\end{definition}
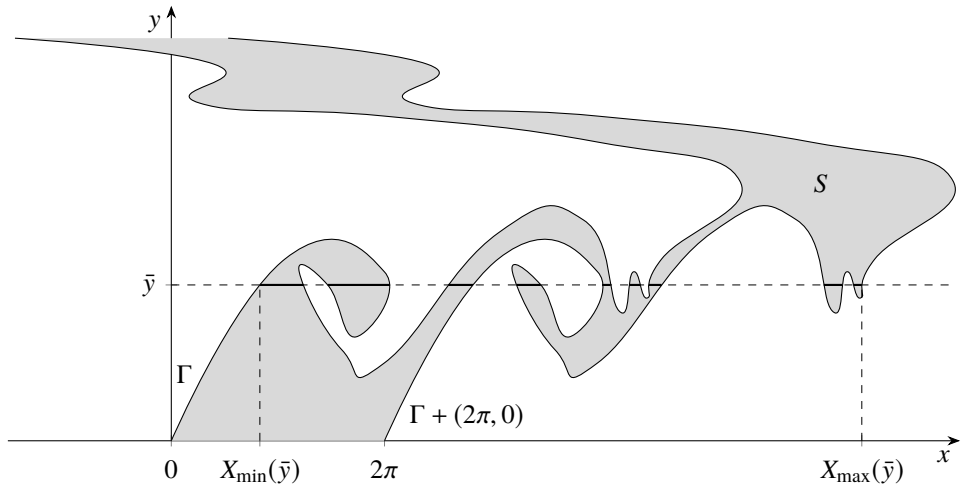
\begin{figure}[htbp]
\centering
\begin{tikzpicture}
\begin{axis}[
    every axis x label/.style={at={(current axis.right of origin)},anchor=north west},
    width=.96\linewidth,
    height=.5\linewidth,
    xmin=-1.15, xmax=5.55,
    ymin=0, ymax=3.35,
    axis y line=middle, axis x line=bottom,
    typeset ticklabels with strut,
    axis line style={-Stealth},
    xlabel=$x$, ylabel=$y$,
    x label style={below left},
    y label style={below left},
    xtick={0,.623,1.5,4.858},
    xticklabels={0,$\Xmin(\bar y)$,$2\pi$,$\Xmax(\bar y)$},
    ytick={1.2},
    yticklabels={$\bar y$},
   ]
\addplot[name path=A, smooth, tension=0.9] coordinates
        {(0,0) (0.8,1.4) (1.5,1.3) (1.3,.8) (1.1,1.2) (.9,1.3)
         (1.2,.7) (1.5,.6) (2.3,1.6)
         (2.9,1.7) (3.15, 1) (3.25, 1.3) (3.35, 1.1) (3.45, 1.4) (4,2)
         (3,2.3) (1.5,2.5) (.2,2.6)
         (.3,2.9) (-1.1,3.1)}
         node[right=25, pos = .61] {$S$}
         node[above left, pos = .05] {$\Gamma$};
\addplot[name path=B, smooth, tension=0.9] coordinates
        {(1.5,0) (2.3,1.4) (3,1.3) (2.8,.8) (2.6,1.2) (2.4,1.3)
         (2.7,.7) (3,.6) (3.8,1.6)
         (4.4,1.7) (4.65, 1) (4.75, 1.3) (4.85, 1.1) (4.95, 1.4) (5.5,2)
         (4.5,2.3) (3,2.5) (1.7,2.6)
         (1.8,2.9) (.4,3.1)}
         node[right, pos = .015] {$\Gamma + (2\pi,0)$};
\addplot[gray!30] fill between[of=A and B];
\addplot[very thin, dashed, domain=0:5.5] {1.2};
\addplot[very thin, dashed, domain=0:1.2] ({.623},x);
\addplot[very thin, dashed, domain=0:1.2] ({4.858},x);
\addplot[thick, domain=.623:.93] {1.2};
\addplot[thick, domain=1.1:1.54] {1.2};
\addplot[thick, domain=1.95:2.123] {1.2};
\addplot[thick, domain=2.43:2.599] {1.2};
\addplot[thick, domain=3.031:3.09] {1.2};
\addplot[thick, domain=3.22:3.3] {1.2};
\addplot[thick, domain=3.36:3.45] {1.2};
\addplot[thick, domain=4.59:4.72] {1.2};
\addplot[thick, domain=4.8:4.858] {1.2};
\end{axis}
\end{tikzpicture}
\caption{Possible behavior of $S$. The thick horizontal
segments correspond to the connected components of $S_{\bar y}$.
\label{fig:strip}}
\end{figure}
We observe that, though in general the diameter of $S_y$ may be arbitrarily large, as well as its
number of connected components, in any case its length is bounded by $2\pi$. This can be easily deduced
recalling that, according to Section \ref{sec:reduction}, $S$ is a covering of the disk minus a
simple curve connecting the boundary to the center. We provide a short, self-contained
proof, which will be useful in the following.
\begin{lemma}\label{lem:S_first_prop}
For every $y$, $|S_y|\leq 2\pi$ and $|\{x:(x,y)\in \Gamma\cup S\}| = 2\pi$.
\end{lemma}
\begin{proof}
We prove the first part of the statement, the second one follows in a similar way. For any $y>0$, we have that $S_y$ is open in the topology of the horizontal line
having height $y$, and therefore it is the disjoint union of at most countable open
intervals:
\begin{equation}\label{eq:propertiesofS}
S_{y} = \bigcup_{i=1}^{k} (s_{2i-1},s_{2i})\times\{\bar y\},\qquad k\leq+\infty.
\end{equation}

Note that if $x_1,x_2\in\overline{S}_{y}$ and
$x_2-x_1=2h\pi$, with $h\in\Z$, then necessarily either $h=0$, or both $(x_i,y)\in\partial S$,
$h=\pm1$; indeed, by translation, we can assume w.l.o.g. that, say, $(x_1,y)\in\partial S$,
$(x_2,y)\in\overline{S}$, so that $(x_1,y)\in\Gamma + (2j\pi,0)$, with $j\in\{0,1\}$,
and
\[
(x_2,y) = (x_1+2h\pi,y)\in \Gamma + (2(j+h)\pi,0),
\]
and this curve does not intersect $\overline{S}$ if $j+h\not\in\{0,1\}$.

We deduce that, for every $i$, there exists $h(i)\in\N$ such that
\[
(\hat s_{2i-1},\hat s_{2i}):=(s_{2i-1}+2h(i)\pi,s_{2i}+2h(i)\pi)\subset(\Xmin(y),\Xmin(y)+2\pi),
\]
and $(\hat s_{2i-1},\hat s_{2i})\cap(\hat s_{2j-1},\hat s_{2j})=\emptyset$ for $i\neq j$.
Thus
\[
\left|S_{y}\right| = \left|\bigcup_{i=1}^{k} (\hat s_{2i-1},\hat s_{2i})\right|
\leq \left|(\Xmin(y),\Xmin(y)+2\pi)\right|=2\pi. \qedhere
\]
\end{proof}
As we mentioned, since $\singset\setminus\{\bm{0}\}$ is empty, \eqref{eq:main_on_the_halfplane} reduces to
\begin{equation}\label{eq:sys_v}
\begin{cases}
\Delta v =0 \quad\text{ in } \R^2_+\\
v=0 \quad\text{ on } \Gamma\\
v(x+2\pi,y)=\lambda v(x,y)\smallskip\\
\int_{S}|\nabla v|^2\,dxdy<+\infty,
\end{cases}
\end{equation}
where $\lambda>0$ is  fixed constant, and we can assume w.l.o.g. that $v$ is of class $C^2$ up to $\{y=0\}$ (possibly replacing it with the restriction
on $\R^2_+$ of $v(x,y+\eps)$, $\eps>0$ small).
\begin{remark}\label{rem:poicare_on_S}
Since $|S_y|\leq2\pi$, and $v$ vanishes at the endpoints of any connected component of $|S_y|$, we
readily infer the validity of a Poincar\'e inequality for $v$ in $S$:
\[
\int_{S_y}v^2\,dx\leq 4 \int_{S_y}|\nabla v|^2\,dx\text{ for every }y,\quad
\int_{S}v^2\,dxdy\leq 4 \int_{S}|\nabla v|^2\,dxdy<+\infty.
\]
Furthermore, we can apply the standard trace inequality on the half plane $\{y\geq\bar y\}$
to the function $v|_{S}$ (with null extension outside $S$) in order to obtain
\[
\int_{S_{\bar y}}v^2\,dx \leq 2\|v|_{S}\|^2_{H^1(\R\times(\bar y,+\infty))}
\leq 10\int_{S\cap\{y>\bar y\}} |\nabla v|^2\,dxdy.
\]
\end{remark}
The aim of this section is to prove the following result.
\begin{proposition}\label{prop:target}
Under the above notation, assume that $v|_S$ has
exactly $2\np$ nodal regions (as well as $v(x,0)|_{(0,2\pi)}$), and let
us define
\begin{equation}\label{eq:alpha}
\alpha:=\frac{\log\lambda}{2\pi}.
\end{equation}
Then there exist constants $q$, $a$, $b$ such that
\[
S_y \subset \left\{(x,y): -\frac{\alpha}{\np}y+q
\leq x \leq -\frac{\alpha}{\np}y+q+2\pi+o(1)\right\},
\]
and
\[
v(x,y) = \left[a \cos{\left(\np x+\alpha y\right)}+
b \sin{\left(\np x+\alpha y\right)} + o(1)\right]
\exp\left(\alpha x-\np y\right)
\]
as $y\to+\infty$, uniformly in $S$.
\end{proposition}
To prove the proposition, the basic idea is to reduce the third condition in
\eqref{eq:sys_v} to a periodic one, and
then to use separation of variables to write the solution in Fourier series. To this aim, for concreteness from now on we assume $\lambda>1$, so that
\[
\alpha>0.
\]
The case $0<\lambda<1$ can be treated in the same way, while the case $\lambda=1$ is actually easier (indeed, before lifting to $\R^2_+$, the function $\tilde v$ defined in \eqref{eq:vtilde}
is already harmonic and bounded in $B_1\setminus\{\bm{0}\}$).
\begin{lemma}\label{fourier}
Let
\[
w(x,y):=e^{-\alpha x}v(x,y).
\]
Then $w$ is $C^2$ in $\overline{\R^2_+}$,  $2\pi$--periodic with respect to the
$x$ variable, and there exist (real) numbers $a_k$, $b_k$, $k\in\Z$, such that
\[
w(x,y)=\sum_{k\in\Z} \left[a_k \cos{(kx+\alpha y)}+b_k
\sin{(kx+\alpha y)}\right] e^{-ky}.
\]
\end{lemma}
\begin{proof}
A direct calculation shows that, for every $(x,y)$,
\[
w(x+2\pi,y) = e^{-\alpha (x + 2\pi)}v(x+2\pi,y)
= \lambda e^{-2\alpha\pi} e^{-\alpha x}v(x,y) = w(x,y).
\]
Since $\Delta (e^{\alpha x} w(x,y))=0$ on the half-plane, $w$ satisfies
\[
\Delta w+2\alpha w_x+\alpha^2 w=0.
\]
Using the periodicity of $w$ we can write
$w(x,y)=\sum_k W_k(y)e^{ikx}$;
substituting we obtain
\[
\sum \left(-k^2W_k + W''_k + 2 ik\alpha W_k + \alpha^2 W_k \right)e^{ikx}
=0
\]
for all $(x,y)$ in the half-space. This implies, for every $k$,
\[
W''_k-(k-i\alpha)^2W_k=0,
\qquad
\text{that is}
\quad
W_k(y)=A_k e^{(k-i\alpha)y}+B_k e^{(-k+i\alpha)y},
\]
where $A_k$, $B_k$ are suitable complex constants. As a consequence
\[
\begin{split}
w(x,y)&=\sum\left[A_k e^{ky}e^{i(kx-\alpha y)} + B_k e^{-ky}e^{i(kx+\alpha y)}\right]\\
&=\sum\left[A_{-k}e^{-i(kx+\alpha y)} + B_k e^{i(kx+\alpha y)}\right]e^{-ky}.
\end{split}
\]
Since $w$ is real, we have that $A_{-k}$ and ${B_k}$ are conjugated,
and the lemma follows by choosing $a_k=A_{-k}+{B_k}$, $b_k=i(-A_{-k}+{B_k})$.
\end{proof}
\begin{corollary}\label{coro:vnicevbad}
If $v$ satisfies \eqref{eq:sys_v} then
\[
v = v_{\mathrm{nice}} + v_{\mathrm{bad}},
\]
where
\begin{equation}\label{fouri}
\begin{split}
 v_{\mathrm{nice}}(x,y) &= \sum_{k=0}^{+\infty} \left[a_k \cos{(kx+\alpha y)}+b_k
\sin{(kx+\alpha y)}\right] e^{\alpha x-ky}, \\
v_{\mathrm{bad}}(x,y) &= \sum_{k=1}^{+\infty} \left[
a_{-k}
\cos{(kx-\alpha y)}
-b_{-k}
\sin{(kx-\alpha y)}\right] e^{\alpha x+ky},
\end{split}
\end{equation}
where the constants $a_k$, $b_k$ are those in Lemma \ref{fourier} (of course, the terminology is due to the fact that
there exist points in $S$ having vertical coordinate $y$ arbitrarily large).
\end{corollary}
Now we want to exploit the further conditions about $v$ to
determine the constants in \eqref{fouri}. Roughly speaking, the
idea is that the condition
$$
\int_{S}|\nabla v|^2\,dxdy<+\infty
$$
should annihilate the ``bad'' part (namely, it should imply that
$a_k=b_k=0$ for every $k<0$), whereas the number of nodal regions of $v$
should determine the dominant part (that is, the first nonzero
$a_k$ and $b_k$). Actually, the first step is not so straight:
indeed, since we do not know the actual position of $S$, it is not trivial
to exclude the integrability on $S$ of quantities of order $e^{2(\alpha x+ky)}$, $k>0$,
even for arbitrarily large $k$.

To start with, we collect in the following lemma some routine consequences
of the theory of Fourier series, when applied to $w$.
\begin{lemma}\label{lem:coeff_to_0}
There exist a constant $C$ (only depending on $\|w\|_{C^2([0,2\pi]\times[0,\pi/(2\alpha)])}$)
such that, for every $k\in\Z$, $k\neq0$,
\begin{enumerate}
 \item $|a_k|+|b_k| \leq \dfrac{C}{k^2}$;
 \item $\displaystyle\pi(a_k^2+b^2_{k})e^{-2ky}<\int_0^{2\pi}w^2(x,y)\,dx + C$.
\end{enumerate}
\end{lemma}
\begin{proof}
Choosing $y=0$ in the expression of $w$ we obtain, for every $k\geq1$,
\[
\left|a_k+a_{-k}\right| = \frac{1}{\pi}\left|\int_0^{2\pi} w(x,0)\cos{kx}\,dx\right| = \frac{1}{\pi k^2 }\left|\int_0^{2\pi} w_{xx}(x,0)\cos{kx}\,dx\right|\leq \frac{C_1}{k^2}.
\]
Analogously, the choice $y=\pi/(2\alpha)$ yields
\[
\left|e^{k\pi/(2\alpha)} a_{-k} - e^{-k\pi/(2\alpha)} a_k\right| = \frac{1}{\pi}\left|\int_0^{2\pi} w\left(x,\frac{\pi}{2\alpha}\right)\sin{kx}\,dx\right|\leq \frac{C_2}{k^2}.
\]
We deduce
\[
|a_k|\leq  \left|a_k + a_{-k} - a_{-k} +  e^{-2k\pi/\alpha} a_k\right| \leq \frac{C_1 + C_2 e^{-k\pi/(2\alpha)} }{k^2}\leq \frac{C_1 + C_2  }{k^2},
\]
\[
|a_{-k}|\leq  \left|a_k + a_{-k} - a_k\right| \leq \frac{2C_1 + C_2  }{k^2}.
\]
One can obtain analogous estimates for $|b_k|$, $|b_{-k}|$ similarly, thus concluding estimate 1.

For the second part, we resume the notation of the proof of Lemma \ref{fourier}.
Applying Parseval's identity we obtain, for every $y\geq0$,
\begin{multline*}
\frac{1}{2\pi}\int_0^{2\pi}w^2(x,y)\,dx = \sum_{k\in\Z}W_k(y)\overline{W_k(y)}\\
 = \sum_{k\in\Z} A_kB_{-k} e^{2ky} +  A_{-k}B_{k} e^{-2ky} +
 \underbrace{A_kA_{-k} e^{-2i\alpha y} + B_kB_{-k} e^{2i\alpha y}
 }_{=2\real A_kA_{-k} e^{-2i\alpha y}} \\
 = \sum_{k\in\Z} \frac{a_{-k}^2+b_{-k}^2}{4} e^{2ky} +
 \frac{a_{k}^2+b_{k}^2}{4} e^{-2ky} \\
 + \frac{a_k a_{-k} - b_{k}b_{-k}}{2} \cos(2\alpha y)
 + \frac{a_k b_{-k} + a_{-k}b_{k}}{2} \sin(2\alpha y),
 \end{multline*}
that is
\[
\int_0^{2\pi}w^2(x,y)\,dx = \pi\sum_{k\in\Z}(a_k^2+b^2_{k})e^{-2ky}
+\Rcal(y),
\]
where
\[
|\Rcal(y)|\leq 2\sum_{k\in\Z} \left(|a_k|+ |b_{k}|\right)\left(|a_{-k}|+|b_{-k}|\right)\leq C,
\qquad\text{for every }y,
\]
by the first part of the lemma.
\end{proof}
The above estimates allow to show that, in case $v_{\mathrm{bad}}\not\equiv0$, then infinitely
many coefficients $a_k,b_k$ must be different from zero.
\begin{lemma}\label{lem:infiniti_cattivi}
Let us assume that
\[
\bar k := \inf\left\{k\in\Z : a_k^2+b_k^2 \neq 0\right\}>-\infty.
\]
Then Proposition \ref{prop:target} follows, and $\bar k = \np >0$ (recall that $2\np$ is the number
of nodal regions of $v(x,0)$ in $[0,2\pi]$).
\end{lemma}
\begin{proof}
From \eqref{fouri}, we infer
that
\begin{multline*}
\left|e^{\bar ky}w(x,y) - \left[a_{\bar k} \cos{(\bar kx+\alpha y)}+
b_{\bar k} \sin{(\bar kx+\alpha y)}\right]\right| \\
=\left|\sum_{k\geq \bar k + 1} \left[a_k \cos{(kx+\alpha y)}+b_k
\sin{(kx+\alpha y)}\right] e^{-(k-\bar k)y}\right|.
\end{multline*}
In particular, since the zero sets of $v$ and $w$ coincide,
\[
(x,y)\in\Gamma
\qquad\implies\qquad
\left|a_{\bar k} \cos{(\bar kx+\alpha y)}+
b_{\bar k} \sin{(\bar kx+\alpha y)}\right|\leq C e^{-y}.
\]
Since $\Gamma$ is connected, we obtain that there exists $q\in\R$ such that
for any $\eps>0$ there exists $\bar y$ large such that
\[
(\Gamma\cap\{y\geq\bar y\}) \subset \{(x,y):q \leq \bar k x+\alpha y \leq q+\eps\} .
\]
The same holds for any nodal line of $v$, in particular for $\Gamma + (2\pi,0)$. Thus we
have shown that, for $y\geq\bar y$, $S$ lies between the two straight
lines of equations $\bar k x+\alpha y=q$, $\bar k x+\alpha y=q+2\pi+\eps$.

Let us assume by contradiction $\bar k<0$. We infer that, for any $(x,y)\in S$ with
$y\geq \bar y$ sufficiently large, $x$ has to be positive. Since $\alpha$ is positive too,
recalling Remark \ref{rem:poicare_on_S} we can write
\[
\begin{split}
+\infty&>\int_{S}v^2\,dxdy\geq
\int_{\bar y}^{+\infty}dy\int_{S_y}e^{2\alpha x }w^2\,dx
\geq
\int_{\bar y}^{+\infty}dy\int_{S_y}w^2\,dx \\ &\geq
\int_{\bar y}^{+\infty}\left[(a_{\bar k}^2+b^2_{\bar k})e^{-2\bar
k y}-C\right]\,dy,
\end{split}
\]
forcing $a_{\bar k}^2+b^2_{\bar k}=0$, in contradiction with the definition of $\bar k$.

Therefore $\bar k\geq0$, and $v=v_{\mathrm{nice}}$ (with summation starting from $k=\bar k$).
Recalling again that $S$ is controlled above and below by straight lines, we finally deduce
that $v$ and the first non-zero term in $v_{\mathrm{nice}}$ are close each other, for $y$ large,
in the $C^2$ norm. But then the number of nodal zones of $v$ in $S_y$, i.e. $2\np$,
must be equal to that of $\cos{(\bar kx+\alpha y)}$, i.e. $2\bar k$, concluding the proof.
\end{proof}
The previous lemma, together with Lemma \ref{lem:coeff_to_0}, suggests that
if the bad part is non-zero then the quantity $\int_0^{2\pi}w^2(x,y)\,dx$, as a function of
$y$, must increase more than exponentially.
\begin{lemma}\label{lem:se decade no cattivi}
If there exist constants $A$, $\beta$ and a sequence $y_n\to+\infty$ such that
\[
\int_0^{2\pi} w^2(x,y_n)\,dx\leq A e^{\beta y_n},
\]
then Proposition \ref{prop:target} follows.
\end{lemma}
\begin{proof}
Combining such assumption with the second estimate in  Lemma \ref{lem:coeff_to_0} we have,
for every $k$ and $n$,
\[
\pi(a_{k}^2+b^2_{k})e^{-2k y_n} \leq A e^{\beta y_n} + C.
\]
Choosing $n$ sufficiently large, this inequality forces $a_k=b_k=0$ whenever $2k<-\beta$,
hence Lemma \ref{lem:infiniti_cattivi} applies.
\end{proof}
Since $v^2=e^{2\alpha x}w^2$ is integrable on $S$, it is easy to see that the assumption
of Lemma \ref{lem:se decade no cattivi} is fulfilled when the strip $S$ lies on the right of a
fixed straight line (it is even trivial if it lies in the
sector $\{x\geq0\}$). In fact, by the trace inequality it is sufficient to assume this property
only for a sequence $(S_{y_n})_n$.
\begin{lemma}\label{lem:no_a_dx}
If there exist a constant $m\in\R$ and a sequence $y_n\to+\infty$ such that
\[
\frac{\Xmin(y_n)}{y_n}\geq m,
\]
then Proposition \ref{prop:target} follows.
\end{lemma}
\begin{proof}
Remark \ref{rem:poicare_on_S} yields, for every $n$,
\[
\begin{split}
\int_0^{2\pi}w^2(x,y_n)\,dx &= \int_{S_{y_n}} e^{-2\alpha x } v^2\,dx \leq
e^{- 2\alpha \Xmin(y_n) }\int_{S_{y_n}}v^2\,dx \\
&\leq e^{- 2\alpha m y_n } \cdot 10\int_{S\cap\{y >  y_n\}}|\nabla v|^2\,dxdy
\end{split}
\]
(recall that $\alpha\geq0$). But then we can conclude by applying
Lemma \ref{lem:se decade no cattivi}.
\end{proof}
At this point, we have ruled out the case when $S$ stays frequently
on the right of some straight line. To face the complementary case, we need the
following result, which can be seen as a one-phase version of the
Alt-Caffarelli-Friedman monotonicity lemma, adapted to this situation.
\begin{lemma}\label{lem:ACF}
Let $t\mapsto (x(t),y(t))$ denote a regular parameterization of $\Gamma$.
For $\bar y>0$ let
\[
t_1(\bar y) := \min\{t:y(t)=\bar y\},\qquad
X_1(\bar y) := x(t_1(\bar y)) + 2\pi,
\]
so that $(X_1(\bar y),\bar y)$ belongs to $\Gamma + (2\pi,0)$ and $\Xmin(\bar y)
\leq X_1(\bar y)\leq  \Xmax (\bar y)$.

If $X_1(\bar y)<0$ then
\[
\int_{S\setminus([X_1(\bar y),2\pi]\times[0,\bar y])}|\nabla v|^2\,dxdy \leq e^{- X_1 ^2(\bar y)/\bar y}
\int_{S}|\nabla v|^2\,dxdy.
\]
\end{lemma}
\begin{proof}
For easier notation, throughout the proof of the lemma we will assume that $v$
is truncated to $0$ in $\R^2_+ \setminus S$.
Furthermore, for $ X_1 (\bar y)< \xi < 2\pi$, we write (see Figure \ref{fig:ACF})
\begin{figure}[htbp]
\centering
\begin{tikzpicture}
\begin{axis}[
    every axis x label/.style={at={(current axis.right of origin)},anchor=north west},
    width=.96\linewidth,
    height=.5\linewidth,
    xmin=-6.15, xmax=3.55,
    ymin=0, ymax=6.35,
    axis y line=center, axis x line=bottom,
    typeset ticklabels with strut,
    axis line style={-Stealth},
    xlabel=$x$, ylabel=$y$,
    x label style={below left},
    y label style={below right},
    xtick={-2.89,-2,0,3},
    xticklabels={$X_1(\bar y)$,$\xi$,0,$2\pi$},
    ytick={4,5},
    yticklabels={$\ \ \eta(\xi)$,$\ \ \bar y$},
    y tick label style={right},
   ]
            \begin{scope}
            \clip (-2,0) -- (-2,4) -- (-1,2.8) -- (3,0) --cycle;
\addplot[name path=A, smooth, pattern=north east lines, pattern color=black!75!yellow,
        tension=0.9] coordinates
        {(0,0) (-3,1) (-5,4) (-5.65,3.85) (-4.9,3.5) (-5.2,3.2) (-5.8,3.8) (-6,4.4) (-5.9,5)
         (-5,4.9) (-3.8,3) (-3.7,3.1) (-5,6)} -- (-5,7) -- (3,7) -- (3,0) --cycle;
            \end{scope}
            \begin{scope}
            \clip (-2.01,0) -- (-2.01,4.1) -- (-1,2.9) -- (3.1,0) --cycle;
\addplot[name path=B, smooth, fill=white, tension=0.9] coordinates
        {(3,0) (0,1) (-2,4) (-2.65,3.85) (-1.9,3.5) (-2.2,3.2) (-2.8,3.8) (-3,4.4) (-2.9,5)
         (-2,4.9) (-.8,3) (-.7,3.1) (-2,6)
        } -- (-2,7) -- (3,7) -- (3,0) --cycle;
            \end{scope}
\addplot[name path=A, smooth, tension=0.9] coordinates
        {(0,0) (-3,1) (-5,4) (-5.65,3.85) (-4.9,3.5) (-5.2,3.2) (-5.8,3.8) (-6,4.4) (-5.9,5)
         (-5,4.9) (-3.8,3) (-3.7,3.1) (-5,6)
        };
\addplot[name path=B, smooth, tension=0.9] coordinates
        {(3,0) (0,1) (-2,4) (-2.65,3.85) (-1.9,3.5) (-2.2,3.2) (-2.8,3.8) (-3,4.4) (-2.9,5)
         (-2,4.9) (-.8,3) (-.7,3.1) (-2,6)
        };
\addplot[gray!10] fill between[of=A and B,
        ];
\addplot[smooth] coordinates {(0,0) (0,2.2)};
\addplot[very thin, dashed, domain=-2:0] {4};
\addplot[very thin, dashed, domain=-6.15:0] {5};
\addplot[very thin, dashed, domain=0:4] ({-2},x);
\addplot[very thin, dashed, domain=0:5] ({-2.89},x);
\addplot[thick, domain=.423:3.205] ({-2},x);
\addplot[thick, domain=3.58:4] ({-2},x);
\draw (-3.5,3.7) node {$S$};
\draw (-1,.87) node[fill=gray!10] {$\Xi(\xi)$};
\draw (-2.2,2.2) node {$\Gamma_{\xi}$};
\end{axis}
\end{tikzpicture}
\caption{Notation for the proof of Lemma \ref{lem:ACF} (the thick vertical
segments correspond to the connected components of $\Gamma_{\xi}$).
\label{fig:ACF}}
\end{figure}
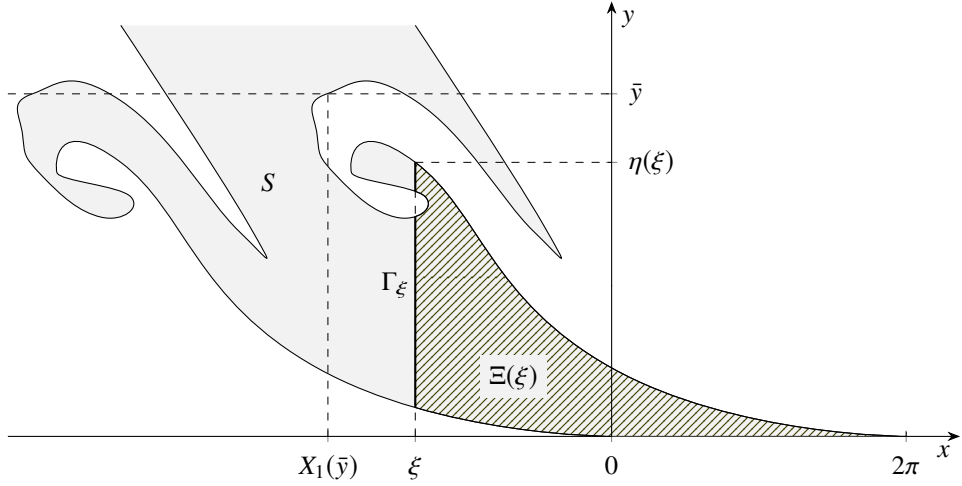%
\[
\begin{split}
\eta(\xi) &:= \max\{y(t):x(t)+2\pi=\xi,\,t<t_1(\bar y)\}<\bar y,
\\
\Gamma_{\xi} &:= \{(\xi,y)\in S: y\leq \eta(\xi)\},
\\
\Xi_{\xi} &:= \bigcup_{\xi\leq x \leq 2\pi} \Gamma_x,
\end{split}
\]
and
\[
\Phi(\xi) := \int_{S\setminus\Xi_{\xi}} |\nabla v|^2\,dxdy
= \int_{\Gamma_{\xi}} vv_x\,dy
\]
(recall that $v$ is harmonic where it is not zero).
Note that
\[
\Phi(\xi) = \Phi( X_1 (\bar y)) + \int_{ X_1 (\bar y)}^{\xi} dx\int_{\Gamma_x} |\nabla v|^2\,dy.
\]
Thus $\Phi$ is absolutely continuous and, since $v=0$ on $\partial \Gamma_{\xi}$,
Poincar\'e inequality implies (for a.e. $\xi$)
\[
\begin{split}
\Phi'(\xi)&= \int_{\Gamma_{\xi}} |\nabla v|^2\,dy =
\int_{\Gamma_{\xi}} v_x^2\,dy + \int_{\Gamma_{\xi}} v_y^2\,dy \geq
\int_{\Gamma_{\xi}} v_x^2\,dy + \frac{\pi^2}{|\Gamma_{\xi}|^2}\int_{\Gamma_{\xi}} v^2\,dy
\\
&\geq
\frac{2\pi}{|\Gamma_{\xi}|}\left(\int_{\Gamma_{\xi}} v^2\,dy\right)^{1/2}
\left(\int_{\Gamma_{\xi}} v_x^2\,dy\right)^{1/2}\geq \frac{2\pi}{|\Gamma_{\xi}|}
\int_{\Gamma_{\xi}} vv_x\,dy=\frac{2\pi}{|\Gamma_{\xi}|}\,\Phi(\xi),
\end{split}
\]
that is
\begin{equation}\label{eq:ACF1}
\frac{\Phi'(\xi)}{\Phi(\xi)} \geq \frac{2\pi}{|\Gamma_{\xi}|}.
\end{equation}
In view of integrating this equation, we can use Jensen's inequality to write
\[
-\frac{1}{\xi} \int_{\xi}^0 \frac{1}{|\Gamma_{x}|}\,dx \geq \frac{1}{
-\frac{1}{\xi} \int_{\xi}^0 |\Gamma_{x}|\,dx} =
\frac{-\xi}{\area(\Xi_{\xi})}\geq
\frac{-\xi}{2\pi\bar y}
\]
(recall that $\Xi_{\xi}\subset S\cap\{y<\bar y\}$ and $|S_y|=2\pi$),
and therefore
\[
\int_{\bar x}^0 \frac{2\pi}{|\Gamma_{x}|}\,dx \geq \frac{\xi^2}{\bar y}.
\]
From \eqref{eq:ACF1} we deduce
\[
\Phi(\xi) \leq e^{-\xi^2/\bar y} \Phi(0) \leq e^{-\xi^2/\bar y} \int_{S}|\nabla v|^2\,dxdy,
\]
and the lemma follows by choosing $\xi =  X_1 (\bar y)$ and recalling that
\[
S\setminus([X_1(\bar y),2\pi]\times[0,\bar y])\quad \subset \quad S\setminus\Xi_{X_1 (\bar y)}.
\qedhere
\]
\end{proof}
The previous lemma allows to treat the case when the diameter of $S_y$ does not
grow too much, for some subsequence.
\begin{lemma}\label{lem:no_diam_bdd}
If there exist a constant $\delta>0$ and a sequence $y_n\to+\infty$ such that
\[
\frac{ X_1 (y_n)-\Xmin(y_n)}{y_n}\leq \delta,
\]
then Proposition \ref{prop:target} follows (of course, the same holds if
$X_1 (y_n)$ is replaced with $\Xmax(y_n)$).
\end{lemma}
\begin{proof}
Recalling Lemma \ref{lem:no_a_dx} we can assume $\Xmin(y_n)/y_n \to - \infty$ so that, in
particular, $ X_1 (y_n)$ must be negative for $n$ large. But then Lemma \ref{lem:ACF}
applies, and we have
\[
\begin{split}
\int_0^{2\pi}w^2(x,y_n)\,dx 
&\leq
e^{- 2\alpha \Xmin(y_n) }\int_{S_{y_n}}v^2\,dx\\
&
\leq 10 e^{- 2\alpha \Xmin(y_n) }\int_{S\cap\{y >  y_n\}}|\nabla v|^2\,dxdy\\
&
\leq C \exp\left(- 2\alpha \Xmin(y_n) - \frac{ X_1 ^2(y_n)}{y_n}\right)\\
&
\leq C \exp\left(\left(2\alpha\delta - 2\alpha \frac{ X_1 (y_n)}{y_n} - \frac{ X_1 ^2(y_n)}{y_n^2}\right)y_n\right)\\
&\leq C \exp\left(\left(2\alpha\delta+\alpha^2\right)y_n\right),
\end{split}
\]
and again we can conclude by Lemma \ref{lem:se decade no cattivi}.
\end{proof}
Summarizing, we are left to consider the case when, for any sequence $y_n\to+\infty$, it holds
\[
\frac{\Xmin(y_n)}{y_n}\to -\infty,\qquad
\frac{ X_1 (y_n)-\Xmin(y_n)}{y_n}\to + \infty.
\]
To this aim, a deeper understanding of the structure of $S_{\bar y}$ is in order, when
${\bar y}$ is large. Even though this analysis can be performed for every ${\bar y}$,
to avoid technicalities we prefer to consider only those values for which the line
$y={\bar y}$ has a finite number of transverse intersections with $\Gamma$ (recall
that, by Sard's Lemma, such condition holds for a.e. $\bar y$).
\begin{lemma}\label{lem:S_further_prop}
Let $t\mapsto (x(t),y(t))$ denote a regular parameterization of $\Gamma$, and
let $\bar y>0$ be a regular value for $t\mapsto y(t)$. Then:
\begin{enumerate}
 \item there exist $t_1<t_2<\dots <t_k$, $k$ odd, such that $y(t)=\bar y$ if and
 only if $t=t_i$ for some $i$;
 \item $(-1)^iy'(t_i)<0$;
 \item writing $S_{\bar y} = \bigcup_{j=1}^{k} (s_{2j-1},s_{2j})\times\{\bar y\}$,
 the intervals being disjoint and ordered, it holds
 \[
 \{s_{2j-1}\}_j=\{x(t_{2i-1}), x(t_{2i}) + 2\pi\}_i,
 \qquad
 \{s_{2j}\}_j=\{x(t_{2i}), x(t_{2i+1}) + 2\pi\}_i;
 \]
 \item $|x(t_{i+1})-x(t_i)|<2\pi$ for every $i=1,\dots,k-1$;
 \item for every $j$, $(s_{2j},\bar y)$ and
 $(s_{2j+1},\bar y)$ both belong either to $\Gamma$ or to $\Gamma + (2\pi,0)$.
\end{enumerate}
\end{lemma}
\begin{proof}
Since $\bar y$ is regular, we immediately deduce that $\Gamma \cap \{y=\bar y\}$ is finite;
thus properties 1 and 2 are trivial consequences of the fact that $y(t)<\bar y$ for $t<t_1$ and
$y(t)>\bar y$ for $t>t_k$.

Observe that $t\mapsto (x(t)+2\pi,y(t))$ parameterize $\Gamma+(2\pi,0)$: we deduce that
the inward normal to $\partial S$ is given by $(y'(t),-x'(t))$ for points of $\Gamma$ and
by $(-y'(t),x'(t))$ for points of $\Gamma+(2\pi,0)$.
By 2, we deduce that if  $i$ is odd (resp. even), then $S_{\bar y}$ lies on the
right (resp. on the left) of $x(t_i)$ and on the left (resp. on the right) of  $x(t_i)+2\pi$, and
also 3 follows.

Let us assume by contradiction that property 4 is false and, say,
\[
  \begin{split}
  x&(t_i)< x(t_{i})+2\pi < x(t_{i+1}) < x(t_{i+1})+2\pi, \\
  &y(t_i)=y(t_{i+1})=\bar y \qquad y(t)\neq \bar y\text{ in }[a,b].
  \end{split}
\]
Let $A$ be the bounded region of $\R^2$, surrounded by the curve joining
$\Gamma|_{(t_i,t_{i+1})}$ and the segment $[x(t_i),x(t_{i+1})]\times\{\bar y\}$.
We deduce that the connected curve $(\Gamma+(2\pi,0))|_{(t_i,t_{i+1})}$ belongs to $A$ for $t\to t_{i}^+$, and to $\R^2\setminus A$ for
$t\to t_{i+1}^+$, so that for some $t^*\in(t_i,t_{i+1})$ it intersects $\partial A$, i.e. $\Gamma$, a contradiction.

In a similar fashion, let 5 be false and, for concreteness, let $i',i''$ be such that $s_{2j} = x(t_{2i'})$ and
 $s_{2j+1}=x(t_{2i''}) + 2\pi$. Let us consider the Jordan curve $\gamma$ obtained by joining
$\Gamma|_{(0,t_{2i'})}$, $(\Gamma+(2\pi,0))|_{(0,t_{2i''})}$ and the two segments
$[s_{2j},s_{2j+1}]\times\{\bar y\}$, $[0,2\pi]\times\{0\}$. Calling $A$ the bounded region
delimited by $\gamma$ we have that $\Gamma|_{(t_{2i'},+\infty)}$ belongs to $A$ for
$t\to t_{2i'}^+$, while it belongs to its complement for $t\to +\infty$. Thus
$\Gamma|_{(t_{2i'},+\infty)}$ must intersect $\gamma$, a contradiction since
$(s_{2j},s_{2j+1})\times\{\bar y\} \cap S_{\bar y} =\emptyset$.
\end{proof}
\begin{remark}\label{rem:prop4}
Property 4 and Property 5 in the above lemma imply that, if $I\subset[\Xmin(\bar y),\Xmax(\bar y)]$ is a closed interval, with $|I|\geq2\pi$,
then it contains an interval $(s_{2j-1},s_{2j})\subset S_{\bar y}$ which has one endpoint on $\Gamma$ and the other on $\Gamma + (2\pi,0)$: indeed,
by 4, $I$ contains at least one point of $\Gamma$ and one point of $\Gamma + (2\pi,0)$, and the claim follows by 5. Moreover,
if $x(t_{i_1})$ and $x(t_{i_2})$ belong to different connected components of $\R\setminus I$, then either $s_{2j-1}$ or $s_{2j}$ belong
to $\Gamma|_{(t_{i_1},t_{i_2})}$.
\end{remark}
\begin{definition}\label{def:bottleneck}
If $S_{\bar y}$ is not connected we call its connected components \emph{bottlenecks}.
A bottleneck
\[
\Sigma = (s',s'')\times\{\bar y\}
\]
disconnects $S$ in two (open, connected) components, only one of which is adjacent to $\{y = 0\}$ (where the non-homogeneous Dirichlet datum is assigned for $v$):
\[
S \setminus \Sigma = \mathcal{N}_\Sigma \cup \Zcal_\Sigma,\qquad \text{with}\qquad
\begin{cases}
\mathcal{N}_\Sigma \cap \Zcal_\Sigma = \emptyset \\
\partial \mathcal{N}_\Sigma \supset [0,2\pi] \times \{0\}\\
v=0\text{ on }\partial\Zcal_\Sigma\setminus\Sigma.
\end{cases}
\]
In the following we will only deal with the component $\Zcal_\Sigma$ ($\Zcal$ stays
for ``Zero
Dirichlet data''), which may be either unbounded (when the endpoints of $\Sigma$
belong to different components of $\partial S$) or not (when they both belong either
to $\Gamma$ or to $\Gamma+(2\pi,0)$).

Finally, we denote with $B_r(\Sigma)$ the disk of radius $r$ centered at the middle point of
$\Sigma$, and with
\[
B'_r(\Sigma) := \text{the connected component of }S \cap B_r(\Sigma)\text{
which contains }\Sigma.
\]
\end{definition}
In fact, any small bottleneck provides a strong decay of the
gradient of $v$, as shown in the following lemma.
\begin{lemma}\label{lem:1strozz}
Let $\pi<R_1<R_2$. Under the notation of Definition \ref{def:bottleneck} it holds
\[
\int_{\Zcal_{\Sigma}\setminus B'_{R_1}(\Sigma)}|\nabla v|^2\,dxdy \leq
\frac{C(R_1,R_2)}{\log(2R_1/|\Sigma|)} \int_{B'_{R_2}(\Sigma)}|\nabla v|^2\,dxdy.
\]
\end{lemma}
\begin{proof}
Let $2\ell=|\Sigma|$ and
\[
\eta(x) =
\begin{cases}
\dfrac{\log |x| - \log \ell}{\log R_1 - \log \ell}  &
x \in \Zcal(\Sigma)\cap B'_{R_1}(\Sigma),\ \ell\leq |x| \leq R_1\\
1  &  x \in \Zcal(\Sigma)\setminus B'_{R_1}(\Sigma)\\
0  &  \text{elsewhere}.
\end{cases}
\]
Note that $\eta^2 v$ is in $H^1_0(\Zcal(\Sigma))$ and $\nabla\eta \not \equiv 0$ only
in $B'_{R_1}(\Sigma)$. Since $v$ is harmonic and $v^2|_{B'_{R_2}(\Sigma)}$ is
subharmonic on $B_{R_2}(\Sigma)$ (when extended to zero outside $B'_{R_2}(\Sigma)$), we obtain
\[
\begin{split}
&\int_{\Zcal_{\Sigma}\setminus B'_{R_1}(\Sigma)}|\nabla v|^2\,dxdy
\leq \int_{\Zcal_{\Sigma}}|\nabla(\eta v)|^2\,dxdy\\
&\qquad
= \int_{\Zcal_{\Sigma}}\nabla v \cdot
\nabla(\eta^2 v)\,dxdy + \int_{\Zcal_{\Sigma}}|\nabla\eta|^2v^2\,dxdy\\
&\qquad \leq
2\pi \int_{\ell}^{1} \frac{r\,dr}{r^2(\log R_1 - \log \ell)^2}
\cdot \left(\max_{\Zcal_{\Sigma}\cap B'_{R_1}(\Sigma)} v^2 \right)\\
&\qquad \leq
\frac{1}{\log R_1 - \log \ell} \cdot \frac{2}{(R_2-R_1)^2}\max_{x\in \Zcal_{\Sigma}\cap
B'_{R_1}(\Sigma)}  \int_{B_{R_2-R_1}(x)} (v|_{B'_{R_2}(\Sigma)})^2\,dxdy\\
&\qquad \leq \frac{1}{\log(2R_1/|\Sigma|)} \cdot \frac{2}{(R_2-R_1)^2}
\int_{B'_{R_2}(\Sigma)}v^2\,dxdy.
\end{split}
\]
To conclude, we recall that there exists $C_P=C_P(R_1,R_2)$ such that the Poincar\'e inequality
\[
 \int_{B_{R_2}}u^2\,dxdy\leq
 C_P\int_{B_{R_2}}|\nabla u|^2\,dxdy
\]
holds, for every $u\in H^1(B_{R_2})$ such that $u|_\gamma=0$, for some connected curve
$\gamma$ having endpoints on $\partial B_{R_1}$ and $\partial B_{R_2}$, respectively (indeed, this implies a Poincar\'e inequality
in  $B_{R_2}\setminus B_{R_1}$, with constant $4/R_2^2$ independent of $\gamma$; then the inequality on $B_{R_2}$ follows using the trace inequality both on
$B_{R_1}$ and $B_{R_2}\setminus B_{R_1}$).
\end{proof}
The above estimate can be used to deal with groups of bottlenecks belonging to localized intervals of $S_{\bar y}$.
\begin{lemma}\label{lem:tantestrozz1intev}
Let $\bar y$ be a regular value of $t\mapsto y(t)$ and $X_1(\bar y)$ be defined as in Lemma \ref{lem:ACF}. For $h \in \N$ we define
\[
\begin{split}
\sigma_h &:= X_1(\bar y) - 3(2h+1)\pi,\\
I_h &:= (\sigma_h - 3\pi,\sigma_h+3\pi)\times\{\bar y\},\\
\delta_h &:= \max\{|\Sigma|:\Sigma\text{ is a bottleneck in }S_{\bar y},\ \Sigma\subset I_h\},\\
B'_h &:= \left\{\bigcup_{i}A_i: A_i \text{ connected component of }S\cap B_{3\pi}(\sigma_h,\bar y)\text{ intersecting }S_{\bar y}\cap I_h
\right\}.
\end{split}
\]
Then there exists a constant $C>0$ such that
\[
\int_{B'_{h+1}} |\nabla v|^2\,dxdy \leq
\frac{C}{\log(3\pi/\delta_{h})} \int_{B'_h}|\nabla v|^2\,dxdy.
\]
\end{lemma}
\begin{proof}
We  recall that, as shown in the proof of Lemma \ref{lem:S_first_prop}, $\overline{S}_{{\bar y}} \cap \{X_1+2j\pi:j\in\Z\} = \{X_1\}$ (for easier notation we drop the dependence
on $\bar y$). As a consequence, each connected component of $S_{\bar y}$ is compactly contained in $(X_1 +2j\pi,X_1+2(j+1)\pi)\times\{\bar y\}$, for some $j\in\Z$.

If $X_1 - 6h\pi < \Xmin$ then there is nothing to prove, thus we can assume $h<(X_1-\Xmin)/(6\pi)$. Recalling the definition of $X_1$, and using Remark \ref{rem:prop4}, we have that
every bottleneck $\Sigma \in I_{h+1}$ belongs to $\Zcal_{\Sigma'}$, for some $\Sigma' \in (\sigma_h - \pi,\sigma_h+\pi)\times\{\bar y\}$.
More formally, writing
\[
S_{\bar y} \cap \{\sigma_h - \pi<x<\sigma_h+\pi\} = \bigcup_{i=1}^{n} \Sigma_i,
\]
we obtain that
\[
B'_{h+1}\ \subset\ B_{3\pi}(\sigma_h,\bar y)\cap \bigcup_{i=1}^{n}\Zcal_{\Sigma_i}.
\]
In order to apply Lemma \ref{lem:1strozz} we choose $R_1 = 3\pi/2$, $R_2=2\pi$; we obtain that, for every $i$,
\[
B'_{R_2}(\Sigma_i) \subset B_{3\pi}(\sigma_h,\bar y), \qquad B'_{R_1}(\Sigma_i)\cap B_{3\pi}(\sigma_{h+1},\bar y) = \emptyset,
\]
and thus
\begin{multline*}
\int_{B'_{h+1}}|\nabla v|^2\,dxdy \leq
\sum_{i=1}^n \int_{\Zcal_{\Sigma_i}\setminus B'_{R_1}(\Sigma_i)}|\nabla v|^2\,dxdy\\ \leq
\sum_{i=1}^n \frac{C}{\log(3\pi/|\Sigma_i|)} \int_{B'_{R_2}(\Sigma_i)}|\nabla v|^2\,dxdy
\leq
\frac{C}{\log(3\pi/\delta_{h})} \int_{B'_h}|\nabla v|^2\,dxdy,
\end{multline*}
for a suitable constant $C>0$.
\end{proof}
Such decay estimate can be easily iterated.
\begin{lemma}\label{lem:kstrozz}
There exists $\gamma>0$ such that, for $1\leq k< (X_1-\Xmin)/(6\pi)$,
\[
\int_{S_{\bar y} \cap  I_k} v^2\,dx \leq
\left(\frac{\gamma}{\log(3k/2)}\right)^k \int_{S\cap\{x < X_1(\bar y)\}}|\nabla v|^2\,dxdy.
\]
\end{lemma}
\begin{proof}
Iterating Lemma \ref{lem:tantestrozz1intev} we obtain
\[
\int_{B'_{k+1}} |\nabla v|^2\,dxdy \leq
\left(\prod_{h=1}^k\frac{C}{\log(3\pi/\delta_{h})}\right) \int_{B'_0}|\nabla v|^2\,dxdy.
\]
Now, using a Poincar\'e inequality analogous to that mentioned at the end of the proof of Lemma \ref{lem:1strozz}, together with
the trace inequality in $B'_{k+1} \cap \{y>\bar y\}$, we obtain the existence of a constant $C_{T}$
such that
\[
\int_{S_{\bar y} \cap  I_k} v^2\,dx \leq C_T  \int_{B'_{k+1}} |\nabla v|^2\,dxdy.
\]
Since $B'_0 \subset \{x < X_1(\bar y)\}$, this implies
\[
\int_{S_{\bar y} \cap  I_k} v^2\,dx \leq
\left(\prod_{h=1}^k\frac{\gamma}{\log(3\pi/\delta_{h})}\right) \int_{S\cap\{x < X_1(\bar y)\}}|\nabla v|^2\,dxdy,
\]
where $\gamma = \max\{C_T,1\} \cdot C$. But then the lemma follows by the elementary estimate
\[
\max\left\{\prod_{h=1}^k\frac{1}{\log(3\pi/\delta_{h})}:\sum_{h=1}^k\delta_{h}\leq 2\pi, \delta_h>0\text{ for every }h\right\} \leq
\left(\frac{1}{\log(3k/2)}\right)^k. \qedhere
\]
\end{proof}
At this point, we have all the ingredients to conclude.
\begin{proof}[End of the proof of Proposition \ref{prop:target}]
As we already mentioned, after Lemmas \ref{lem:infiniti_cattivi}, \ref{lem:no_a_dx} and \ref{lem:no_diam_bdd},
we can assume that, along the  sequence $y_n\to+\infty$ of regular values of $t\mapsto y(t)$, it holds
\[
\frac{\Xmin(y_n)}{y_n}\to -\infty,\qquad
\frac{ X_1 (y_n)-\Xmin(y_n)}{y_n}\to + \infty.
\]
Let $k_n \in\N$, $k_n\to+\infty$, be defined by
\[
\frac{X_1(y_n)-\Xmin(y_n)}{6\pi} -1 < k_n < \frac{X_1(y_n)-\Xmin(y_n)}{6\pi}.
\]
We have ($X_{1,n} := X_1(y_n)$)
\begin{multline*}
\int_0^{2\pi}w^2(x,y_n)\,dx 
= \int_{S_{y_n}\cap\{x > X_{1,n}\}} e^{-2\alpha x } v^2\,dx + \sum_{h=1}^{k_n} \int_{S_{y_n} \cap  I_h}e^{-2\alpha x }  v^2\,dx  \\
\leq \underbrace{e^{-2\alpha X_{1,n} } \int_{S_{y_n}\cap\{x > X_{1,n}\}} v^2\,dx}_{(A)} + \underbrace{\sum_{h=1}^{k_n} e^{-2\alpha (X_{1,n} - 6(h+1)\pi) } \int_{S_{y_n} \cap  I_h}  v^2\,dx}_{(B)}.
\end{multline*}
The first term can be easily estimated, as usual, using Remark \ref{rem:poicare_on_S}:
\[
(A) \leq
e^{- 2\alpha X_{1,n} } \int_{S_{y_n}}v^2\,dx
\leq 10 e^{- 2\alpha X_{1,n} }\int_{S\cap\{y >  y_n\}}|\nabla v|^2\,dxdy.
\]
On the other hand, we can bound the second term using Lemma \ref{lem:kstrozz}:
\[
\begin{split}
(B) & \leq
\sum_{h=1}^{k_n} e^{-2\alpha (X_{1,n} - 6(h+1)\pi) } \left(\frac{\gamma}{\log(3h/2)}\right)^h \int_{S\cap\{x < X_{1,n}\}}|\nabla v|^2\,dxdy \\
& \leq
e^{-2\alpha (X_{1,n} - 6\pi) } \int_{S\cap\{x < X_{1,n}\}}|\nabla v|^2\,dxdy \\
&\hspace{5em} \times \sum_{h=1}^{k_n} \exp\left[12h\pi + h \left(\log\gamma-\log \log \frac{3h}{2}\right)\right]  \\
& \leq
e^{-2\alpha (X_{1,n} - 6\pi) } \int_{S\cap\{x < X_{1,n}\}}|\nabla v|^2\,dxdy\\
&\hspace{5em}  \times \underbrace{\sum_{h=1}^{\infty} \exp\left[\left(12\pi + \log\gamma-\log \log \frac{3h}{2}\right)h\right]}_{\leq C}  \\
& \leq
C e^{-2\alpha X_{1,n} } \int_{S\cap\{x < X_{1,n}\}}|\nabla v|^2\,dxdy .
\end{split}
\]
Now, there are two possibilities: either (up to subsequences) $X_{1,n} \geq 0$, in which case
\[
\int_0^{2\pi}w^2(x,y_n)\,dx  \leq (A)+ (B) \leq \int_{S}|\nabla v|^2\,dxdy < +\infty,
\]
and we can conclude using Lemma \ref{lem:se decade no cattivi}; otherwise, if $X_{1,n} < 0$ we can use Lemma \ref{lem:ACF}, obtaining
\[
\begin{split}
\int_0^{2\pi}w^2(x,y_n)\,dx  & \leq (A)+ (B) \leq C e^{-2\alpha X_{1,n} } \int_{S\setminus([X_1(\bar y),2\pi]\times[0,\bar y])}|\nabla v|^2\,dxdy \\
& \leq C \exp\left(-2\alpha X_{1,n} - \frac{X_{1,n}^2}{y_n}\right) \leq C \exp\left(\alpha^2 y_n\right),
\end{split}
\]
and again, by  Lemma \ref{lem:se decade no cattivi}, the proposition follows.
\end{proof}

\section{High multiplicity points are isolated}\label{sec high}
In this section we go back to the general case of system \eqref{eq:main_on_the_halfplane},
aiming at excluding that $\singset$ contains more than one point (one can stick to
the case of an even number of nodal regions, even though all the arguments hold also in the odd
case, taking into account Remark \ref{rem:oddk}). This will allow to complete the proof of Theorem \ref{thm topologia sing}
\begin{proposition}\label{prop:sing are point}
Under the reduction of Section \ref{sec:reduction}, we have that $\singset=\{\bm{0}\}$.
\end{proposition}
By contradiction, throughout the section we assume that the above statement is false. Nonetheless, recall that by construction we can
bound the set $\singset$ into a strip of width $\pi$.
\begin{lemma}\label{lem:confinedW}
Assume that $\singset\supsetneq\{\bm{0}\}$. Then $\mathcal{T}^{-1}(\singset\setminus\{\bm{0}\})$
is the union of infinitely many connected, simply connected, components. Moreover there exists $a>0$
such that, up to a translation, we can assume that any of such components is contained in
$[-\pi/2+2m\pi,\pi/2+2m\pi]\times[a,+\infty)$, for some $m\in\Z$.
\end{lemma}
\begin{proof}
The lemma follows by the definition of the map $\mathcal{T}$ (see equation \eqref{eq:defT}),
together with Proposition \ref{prop:reduction}, and in particular property
\eqref{eq:Winthehalfball} of the model case scenario (MCS).
\end{proof}
\begin{definition}
We denote the unique lifting of $\Gamma_{ij}$ (i.e. the connected component of $\mathcal{T}^{-1}(\Gamma_{ij})$), with an endpoint in $[2m\pi,2(m+1)\pi)\times\{0\}$, $m\in\Z$, as
\[
\tilde\Gamma_{ij}^m,
\]
as long as it is non-empty (i.e., if $i$ and $j$ are consecutive). Analogously, for any $m\in\Z$,
we introduce the set
\[
S^m = S + (2m\pi, 0)
\]
(recall that the strip $S$ has been defined in equation \eqref{eq:defS}).
\end{definition}
To prove the proposition we have to distinguish two cases, according to the horizontal behaviour of the smooth curves $\tilde\Gamma_{ij}^m$.
\begin{lemma}\label{lem:tuttiperuno}
Consider the families of infima, depending on $i$, $j$, $m$,
\[
\inf\{x:(x,y)\in\tilde\Gamma_{ij}^m\},\qquad \inf\{x:(x,y)\in S^m\}.
\]
Then, if one of them is finite, each of them is. An analogous statement holds
for the suprema $\sup\{x:(x,y)\in\tilde\Gamma_{ij}^m\}$, $\sup\{x:(x,y)\in S^m\}$.
\end{lemma}
\begin{proof}
The lemma easily follows from the trivial property
\[
\begin{split}
\inf\{x:(x,y)\in \tilde\Gamma_{k1}^m\}&=\inf\{x:(x,y)\in S^m\} \leq \inf\{x:(x,y)\in\tilde\Gamma_{ij}^m\} \\ &\leq \inf\{x:(x,y)\in\tilde\Gamma_{k1}^{m+1}\} = \inf\{x:(x,y)\in\tilde\Gamma_{k1}^m\} + 2\pi. \qedhere
\end{split}
\]
\end{proof}
\textbf{Case 1)} Both $\inf\{x:(x,y)\in\tilde\Gamma_{ij}^m\}$ and
$\sup\{x:(x,y)\in\tilde\Gamma_{ij}^m\}$ are finite, for every $i,j,m$.

Let $\mathcal{Q} = (-\pi,\pi) \times \R_+$ (notice that $\partial\mathcal{Q}\cap\singset=\emptyset$, by Lemma \ref{lem:confinedW}). We can cover the set $\mathcal{Q} \setminus \zeroset$ with a finite number of copies of $S$, that is, there exists $M \in \N$ such that
\[
	\mathcal{Q} \setminus \mathcal{T}^{-1}(\zeroset \setminus \{0\}) \subset \bigcup_{m = -M}^{M} S^{m}.
\]
On $\mathcal{Q}$ we introduce the $k(2M+1)$ functions $v^{m}_i \in Lip(\mathcal{Q})$, defined for $i=1,\dots,k$ and $m=-M,\dots,M$ as
\[
	v^m_i(x,y) = \begin{cases}
		|v(x,y)| &\text{if $(x,y) \in S^m$, $\mathcal{T}(x,y)\in\omega_i$}\\
		0 &\text{otherwise.}
	\end{cases}
\]
Using the notation of \cite{MR2984134}, we want to show that the vector of such densities
belongs to the class $\mathcal{G}(\mathcal{Q})$ defined in that paper. More precisely,
in the present context this is equivalent to the following lemma.
\begin{lemma}\label{lem:classG}
Let $v^m_i$ be defined as above. Then
\begin{itemize}
	\item each $v^m_i$ is a non-negative, locally Lipschitz, subharmonic function and there exists $\mu^m_i \in \mathcal{M}(\mathcal{Q})$ non-negative Radon measure, supported on $\mathcal{Q} \cap \partial \{ v^m_i > 0 \}$ such that
	\[
		-\Delta v^m_i = - \mu^m_i \quad \text{in the sense of distributions $\mathcal{D}'(\mathcal{Q})$;}
	\]
	\item for every $\bx_0 \in \mathcal{Q}$ and $0 < r < \dist(\bx_0, \partial \mathcal{Q})$ the following identity holds
	\[
		\frac{d}{dr} \int_{B_r(\bx_0)}  \sum_{ i,m } |\nabla v^m_i|^2 = 2 \int_{\partial B_r(\bx_0)} \sum_{ i,m } (\partial_\nu v^m_i) ^2 d \sigma.
	\]
\end{itemize}
\end{lemma}
\begin{proof}
The first point follows directly from the definition of the involved functions, while for the second one we adopt the same strategy of \cite[Theorem 15]{MR2529504} and \cite[Theorem 8.4]{MR2984134}. We consider $\delta > 0$ as a small quantity and we integrate the equation in $v^m_i$ against $(\bx-\bx_0) \cdot \nabla v^m_i$ on the set $B_r(\bx_0) \cap \{v^m_i > \delta\}$. Some integrations by parts (the same one exploits in order to prove the Pohozaev identity) yield
\begin{multline*}
	  \int_{\partial B_r(\bx_0)\cap \{v^m_i > \delta\} } |\nabla v^m_i|^2 = 2 \int_{\partial B_r(\bx_0)\cap \{v^m_i > \delta\}} \sum_{ i,m } (\partial_\nu v^m_i) ^2 d \sigma \\ + \frac{1}{r} \int_{B_r(\bx_0)\cap \partial \{v^m_i > \delta\} } |\nabla v^m_i|^2 (\bx-\bx_0) \cdot \nu d \sigma
\end{multline*}
and thus
\begin{multline*}
	\frac{d}{dr} \int_{B_r(\bx_0)}  \sum_{ i,m } |\nabla v^m_i|^2 =  2 \int_{\partial B_r(\bx_0)} \sum_{ i,m } (\partial_\nu v^m_i) ^2 d \sigma \\+ \frac{1}{r} \lim_{\delta \to 0^+}  \sum_{ i,m } \int_{B_r(\bx_0)\cap \partial \{v^m_i > \delta\} } |\nabla v^m_i|^2 (\bx-\bx_0) \cdot \nu d \sigma.
\end{multline*}

To evaluate the last limit, let $\eps > 0$ and let us consider the set $S_\eps = \{ x \in \mathcal{Q} : \sum_{ i,m } |\nabla v^m_i| \leq \eps\}$. Thanks to Lemma \ref{thm:intro_S} we have that
\[
	 \mathcal{Q} \cap  \mathcal{T}^{-1}(\singset \setminus \{0\}) \subset \mathcal{Q} \cap  S_\eps \qquad \text{for all $\eps > 0$}.
\]
We also observe that, since each $v^m_i$ is harmonic when positive,
\[
	 \int_{B_r(\bx_0) \cap \partial \{v^m_i > \delta \} } |\nabla v^m_i| d \sigma = -\int_{B_r(\bx_0) \cap \partial \{v^m_i > \delta \} } \partial_\nu v^m_i 
	 = \int_{\partial B_r(\bx_0) \cap \{v^m_i > \delta \} } \partial_\nu v^m_i \leq C
\]
for a constant $C$ independent of $\delta$.

We split the remainder into two parts. Outside of $S_\eps$ we can exploit the definition of $v^m_i$ and the regularity of the zero set (which is given by the union of the curves $\tilde \Gamma_{ij}^m$), and find
\begin{multline*}
	\lim_{\delta \to 0^+}  \sum_{ i,m } \int_{B_r(\bx_0)\cap \partial \{v^m_i > \delta\} \setminus S_\eps} |\nabla v^m_i|^2 (\bx-\bx_0) \cdot \nu d \sigma  \\
	=  \sum_{ i,m } \int_{B_r(\bx_0)\cap \partial \{v^m_i > 0\} \setminus S_\eps} |\nabla v^m_i|^2 (\bx-\bx_0) \cdot \nu d \sigma \\
	= \sum_{i,j,m} \int_{B_r(\bx_0)\cap \tilde \Gamma_{ij}^m   \setminus S_\eps} |\nabla v|^2 (\bx-\bx_0) \cdot J d \bm{s} \\
	- \sum_{i,j,m} \int_{B_r(\bx_0)\cap \tilde \Gamma_{ij}^m   \setminus S_\eps} |\nabla v|^2 (\bx-\bx_0) \cdot J d \bm{s} = 0.
\end{multline*}
Here $J$ is the symplectic matrix and the two opposite contributions follows by the fact that each $\tilde \Gamma_{ij}^m$ appears as the boundary of the positivity set of two functions $v^m_i$. On the other hand, inside of $S_\eps$, there exists $C>0$ independent of $\delta$ such that
\begin{multline*}
	\int_{B_r(\bx_0)\cap \partial \{v^m_i > \delta\} \cap S_\eps} |\nabla v^m_i|^2 (\bx-\bx_0) \cdot \nu d \sigma \\
	\leq C \eps \int_{B_r(\bx_0)\cap \partial \{v^m_i > \delta\} \cap S_\eps} |\nabla v^m_i|  d \sigma \leq C \eps.
\end{multline*}
From the arbitrariness of $\eps$ we conclude
\[
	\lim_{\delta \to 0^+}  \sum_{ i,m } \int_{B_r(\bx_0)\cap \partial \{v^m_i > \delta\} } |\nabla v^m_i|^2 (\bx-\bx_0) \cdot \nu d \sigma = 0.\qedhere
\]
\end{proof}
\begin{proof}[Proof of Proposition \ref{prop:sing are point} (Case 1)] By Lemma \ref{lem:classG},
we are in a  position to apply the results of \cite[Theorem 16]{MR2529504},
\cite[Theorem 4.7]{MR2393430} and \cite[Theorem 1.1]{MR2984134}. We find in particular that
the set $\mathcal{T}^{-1}(\singset \setminus \{0\}) \cap \mathcal{Q}$ is made of (at most countable
many) isolated points. Since $\mathcal{W}$ is a connected set, it must be that $\mathcal{T}^{-1}
(\singset \setminus \{0\}) \cap \mathcal{Q}$ contains at most a point $\tilde \bx_0 \in \mathcal{Q}$. But then, going back to the original coordinates, we find that
\[
	\mathcal{W} = \{0\} \cup \{\mathcal{T}(\tilde \bx_0) \}.
\]
Again by the connectedness of $\mathcal{W}$, we conclude that $\mathcal{W} = \{0\}$.
\end{proof}
\medskip

\textbf{Case 2)} Either $\inf\{x:(x,y)\in\tilde\Gamma_{ij}^m\}$ or
$\sup\{x:(x,y)\in\tilde\Gamma_{ij}^m\}$ are not finite, for every $i,j,m$.

Also in this case we write $\mathcal{Q} = (-\pi, \pi)\times (0,+\infty)$ (we recall again that $\partial\mathcal{Q}\cap\singset=\emptyset$, by Lemma \ref{lem:confinedW}).
Possibly up to a further translation and using Sard's Lemma, we can assume that
any intersection between each $\tilde\Gamma_{ij}^m$ and $\partial\mathcal{Q}$
is transverse. Furthermore, we denote by $\concomp_i$, $i \in \N$, the (open) connected components of $\mathcal{Q} \setminus \{v=0\}$.
\begin{lemma}\label{lem v is H1}
If $v$ solves \eqref{eq:main_on_the_halfplane} then $v \in H^1(\mathcal{Q})$.
\end{lemma}
\begin{proof}
Notice that we already know that $v \in H^1(S^m)$, for every $m\in\Z$, and $v \in H^1(\concomp_i)$, for every $i \in \N$. More precisely, there exists a sequence $\{m_i\}_{i \in \N} \subset \Z$ such that
\[
	\concomp_i \subset S^m \iff m = m_i
\]
and, consequently, we have
\[
	v(x,y)|_{\concomp_i} = v(x+2m_i\pi,y)|_{\concomp_i - (2m_i \pi,0)} =  \lambda^{m_i} v(x,y)|_{\concomp_i - (2m_i \pi,0)},
\]
where now each $\concomp_i - (2m_i \pi,0)$ is a subset of $S^0$ and, moreover, for each $i \neq j$, $\concomp_i - (2m_i \pi,0) \cap \concomp_j - (2m_j \pi,0) = \emptyset$. Let us define, for every $m\in\Z$, the 
(possibly empty) sets
\[
E_m := \bigcup_{m_i = m} \concomp_i,
\]
so that
\[
	\int_{E_m}|\nabla v |^2 = \lambda^{2m} \int_{E_m - (2m \pi,0)} |\nabla v|^2.
\]
With this notation, for any $n \in \N$ we define the sequence of functions
\[
	v_n = \sum_{m = -n}^n v |_{E_m} = \sum_{m = -n}^n \lambda^{m}v |_{E_m - (2m \pi,0)}.
\]
By definition, clearly we have that $v_n \in H^1(\mathcal{Q})$. We want to show that the limit of the sequence, that is the function $v$ itself, is a $H^1(\mathcal{Q})$ function. To this end, it is sufficient to show that the series of the norm
\[
	\int_{\mathcal{Q}} |\nabla v|^2 + v^2 = \lim_{n \to +\infty}  \int_{\mathcal{Q}} |\nabla v_n|^2 +v_n^2 = \sum_{m = -\infty}^{+\infty} \int_{E_m}  |\nabla v|^2 +v^2
\]
converges. Here we shall only address the contribution regarding the $L^2$-norm of gradient, as the $L^2$-norm of the functions themselves can be estimated by means of the Poincar\'e inequality (see
Remark \ref{rem:poicare_on_S}, or Lemma \ref{lem:1strozz}).

Since the general term in the series is positive, we can write
\[
	\sum_{m = -\infty}^{+\infty} \int_{E_m}  |\nabla v|^2 = \sum_{m \leq 1} \int_{E_m}  |\nabla v|^2 + \sum_{m \geq 2} \int_{E_m}  |\nabla v|^2.
\]
In the following we deal with the case $\lambda \geq 1$, the opposite case following with minor changes. Consequently, we can estimate the first term as
\[
	\sum_{m\leq 1} \lambda^{2m} \int_{E_m - (2m \pi,0)} |\nabla v|^2 \leq  \lambda^2\sum_{m\leq 1} \int_{E_m - (2m \pi,0)} |\nabla v|^2 \leq  \lambda^2\int_{S} |\nabla v|^2 < +\infty,
\]
in such a way that we are left to consider only the terms with $m\geq 2$. We now have two alternatives. If only a finite number of $E_m$ are non-empty, for $m\geq 2$, then $v \in H^1(\mathcal{Q})$ since $v \in H^1(S^0)$.
Otherwise, let $\bar m \geq 2$ be such that $E_{\bar m}\neq \emptyset$. By definition, it means that the set $E_{\bar m} - (2\bar m \pi,0)$ belongs to $\mathcal{Q}-(2\bar m \pi,0)$. Up to a further translation, we can assume that $\{x = 0\} \cap
\mathcal{T}^{-1}(\singset\setminus\{\bm{0}\}) \neq \emptyset$ and we denote
\begin{itemize}
 \item $y_0 := \min\{y:(0,y)\in\mathcal{T}^{-1}(\singset\setminus\{\bm{0}\})\}$,
 \item $\sigma^m:= \{(2m\pi,y):0\leq y\leq y_0\}$ for $m \in \Z$.
\end{itemize}
In this setting we have that, while $\R^2_+ \setminus \mathcal{T}^{-1}(\singset \setminus \{0\})$ is
connected,
\[
	\R^2_+ \setminus \left(\sigma^m \cup \mathcal{T}^{-1}(\singset \setminus \{0\}) \right)
\]
is disconnected, for every $m$. We deduce that, for any $2 \leq m < \bar m$, the segment $\sigma^m$ disconnects $S^0$ in at least two connected components, one containing $S^0 \cap \mathcal{Q}$, the other $E_{\bar m} - (2\bar m \pi,0)$. We are then in a similar position to that of Definition \ref{def:bottleneck}, with the only difference that now we have to consider vertical \emph{bottlenecks}. With this language, $S^0$ goes through at least $\bar m-1$ bottlenecks before
reaching $E_{\bar m} - (2\bar m \pi,0)$. Using the same proof of Lemma \ref{lem:kstrozz} we have that,
for $\bar m$ large enough,
\[
	\int_{E_{\bar m} -  (2\bar m \pi,0)} |\nabla v|^2 \leq \left(\frac{C}{\log C (\bar m -1)}\right)^{\bar m-1} \int_{S}|\nabla v|^2
\]
where the constant $C > 1$ is universal. Consequently we have
\[\begin{split}
	\sum_{m \geq 2} \int_{E_m}  |\nabla v|^2 &\leq \sum_{m \geq 2} \lambda^{2m} \int_{E_m - (2m \pi,0)} |\nabla v|^2\\
     &\leq \sum_{m \ge 2} \lambda^{2m} \left(\frac{C}{\log  C(m-1)}\right)^{m-1} \int_{S} |\nabla v|^2 < +\infty. \qedhere
\end{split}\]
\end{proof}
To conclude, we want to show that $v$ is harmonic across the singular set.
The key step in this direction is that $v$ can be integrated by parts
on each of its nodal connected components.
\begin{lemma}\label{lem:integrbypartsinD}
For every $i$,
\[
\partial \concomp_i \cap \mathcal{T}^{-1}(\singset\setminus\{\bm{0}\})=\emptyset.
\]
\end{lemma}
\begin{proof}
We are dealing with Case 2. By Lemma \ref{lem:tuttiperuno} we know that $\overline{\mathcal{Q}} \cap \mathcal{T}^{-1}
(\zeroset_2)$ consists in the countable union of smooth paths of finite lengths, pairwise disjointed.
Moreover, infinitely many of them have one endpoint on the line $\{x=-\pi\}$ and the other on
$\{x=\pi\}$. Let us denote with
$\gamma_j$, $j\in\N$ such paths. We recall that $\singset$ is connected and $\bm{0}\in\singset$,
so that any connected component of $\mathcal{T}^{-1}(\singset \setminus \{0\})$ is unbounded.
We deduce that, for every $j$, there exist two
open connected sets $\mathcal{Q}_j^\pm$ satisfying
\[
\mathcal{Q}\setminus\gamma_i = \mathcal{Q}_j^- \cup \mathcal{Q}_j^+,\qquad
\mathcal{Q}_j^-\cap\mathcal{Q}_j^+=\emptyset, \qquad
\overline{\mathcal{Q}_j^-}\cap\mathcal{T}^{-1}(\singset \setminus \{0\}) = \emptyset.
\]
With this notation, we are left to prove that any $\concomp_i$ is contained in some $\mathcal{Q}_j^-$.
Notice that, up to a relabelling, one can assume that
$\mathcal{Q}_j^-\subset\mathcal{Q}_{j+1}^-$, $\mathcal{Q}_j^+\supset\mathcal{Q}_{j+1}^+$. Let us consider the open set $\mathcal{Q}^-=\bigcup_j \mathcal{Q}_j^-$. Since
$\singset \setminus \{0\}$ is non-empty, we have that $\mathcal{Q}^-\neq \mathcal{Q}$,
and by construction
\[
\partial\mathcal{Q}^-\setminus\partial\mathcal{Q} = \lim_j\gamma_j
\subset\mathcal{T}^{-1}(\singset \setminus \{0\}).
\]
Recalling Lemma \ref{lem:confinedW} we infer that
\[
\mathcal{Q}^+:=\bigcap_j \mathcal{Q}_j^+ \subset \left[-\frac{\pi}{2},\frac{\pi}{2}\right]\times[a,+\infty).
\]
Now, take any $\concomp_i$. Of course, for every $j$, either $\concomp_i\subset\mathcal{Q}_j^-$ or
$\concomp_i\subset\mathcal{Q}_j^+$; on the other hand, $\partial \concomp_i\cap\partial\mathcal{Q}
\neq0$. We deduce that $\concomp_i\not\subset\mathcal{Q}^+$, or equivalently that
$\concomp_i\subset\mathcal{Q}_j^-$ for some $j$.
\end{proof}
\begin{lemma}\label{lem:case2harmonicity}
The function $v$ is harmonic in $\mathcal{Q}$ (i.e., it is harmonic in $\R^2_+$).
\end{lemma}
\begin{proof}
Let $\varphi\in C^\infty_0(\mathcal{Q})$, $\varphi\geq0$, and let us write
\[
\supp\varphi\cap\mathcal{T}^{-1}(\zeroset_2)=\bigcup_{j=0}^{+\infty}\gamma_j.
\]
By Lemmas \ref{lem v is H1} and \ref{lem:integrbypartsinD},
we have that
\begin{equation}\label{eq:perparti}
	\int_\mathcal{Q} \nabla v \cdot\nabla \varphi = \sum_{i = 0} ^{+\infty} \int_{D_i} \nabla v \cdot
	\nabla \varphi = \sum_{i = 0} ^{+\infty} \int_{\partial \concomp_i} \varphi\partial_\nu v
	=\sum_{i = 0} ^{+\infty}\sum_{j:\gamma_j \subset \partial\concomp_i}  \int_{\gamma_j} \varphi\partial_{\nu_i} v,
\end{equation}
where $\nu_i$ denotes the normal direction to $\gamma_j$ which points outwards with respect to
$D_i$. Of course, for every $j$ there exist exactly two indexes $i_1$, $i_2$ such that
$\gamma_j = \overline{\concomp}_{i_1} \cap \overline{\concomp}_{i_2} \cap \supp\varphi$, and
$\nu_{i_1} = -\nu_{i_2}$;
as a consequence, for every $j$,
\begin{equation}\label{eq:0is0}
\sum_{i:\partial\concomp_i\supset \gamma_j }  \int_{\gamma_j} \varphi\partial_{\nu_i} v =
\int_{\gamma_j} \varphi\partial_{\nu_{i_1}} v + \int_{\gamma_j}
\varphi\partial_{\nu_{i_2}} v = 0.
\end{equation}
In order to plug this relation into \eqref{eq:perparti} and conclude the proof, we need to
show that the right hand side of \eqref{eq:perparti} converges absolutely,
so that we are allowed to rearrange its terms. To this aim let us notice that, since each $D_i$ is a
nodal region of $v$, and $\varphi$ is non-negative (and compactly supported in $\mathcal{Q}$), we have that $\varphi\partial_\nu v$ does not
change its sign on $\partial D_i$. This yields, for every $i$,
\begin{multline*}
\sum_{j:\gamma_j \subset \partial\concomp_i}  \left|
\int_{\gamma_j} \varphi\partial_{\nu_i} v\right| =
\left| \sum_{j:\gamma_j \subset \partial\concomp_i}
\int_{\gamma_j} \varphi\partial_{\nu_i} v\right| \\
=\left|\int_{\partial \concomp_i} \varphi\partial_\nu v\right|=
 \left|\int_{D_i} \nabla v \cdot\nabla \varphi\right| \leq
 \int_{D_i} |\nabla v| |\nabla \varphi| ,
\end{multline*}
and finally
\[
\sum_{i,j:\gamma_j \subset \partial\concomp_i}  \left|
\int_{\gamma_j} \varphi\partial_{\nu_i} v\right|
 \leq \| v\|_{H^1(\mathcal{Q})}
 \| \varphi\|_{H^1(\mathcal{Q})}<+\infty
\]
by Lemma \ref{lem v is H1}. Therefore, combining \eqref{eq:perparti} and \eqref{eq:0is0}
we have, for every $\varphi\in C^\infty_0(\mathcal{Q})$, $\varphi\geq0$,
\[
\int_\mathcal{Q} \nabla v \nabla \varphi = 0. \qedhere
\]
\end{proof}
\begin{proof}[Proof of Proposition \ref{prop:sing are point} (Case 2)]
Since the zeroes of higher multiplicity of a (non trivial) harmonic function in the plane are
isolated, using Lemma \ref{lem:case2harmonicity} we find again that
the set $\mathcal{T}^{-1}(\singset \setminus \{0\}) \cap \mathcal{Q}$ is made of (at most countable
many) isolated points, and we can conclude as in Case 1.
\end{proof}

\begin{proof}[Proof of Theorem \ref{thm topologia sing}]
By Proposition \ref{prop:reduction} we have that the singular set $\singset$ consists in the finite
union of connected components, which are single points by Proposition \ref{prop:sing are point},
therefore the only thing that is left to prove is the asymptotic expansion \eqref{eq:asympt_expans}. Let $m(\bx_0)=h\ge3$, and let $\mathcal{U}$, $v$ be defined as in \eqref{eq:vtilde}, \eqref{eq:defvi},
respectively. Taking into account Proposition \ref{prop:target}, together with Remarks \ref{rem:asy1} and \ref{rem:Tallamenouno}, we obtain
\[
\begin{split}
v(x,y) &= \exp\left(\alpha x-\frac{h}{2} y\right) \times
\left[a \cos{\left(\frac{h}{2} x+\alpha y\right)}+
b \sin{\left(\frac{h}{2} x+\alpha y\right)} + o(1)\right]\\
&= C \exp\left(\alpha x + \frac{2\alpha^2}{h}y \right) \times \exp\left(-\frac{2\alpha^2}{h}y-\frac{h}{2} y\right)\\
&\hspace{10.2em} \times \left[ \cos{\left(\frac{h}{2} (x-x_0) + \alpha y\right)} + o(1)\right].
\end{split}
\]
Recalling \eqref{eq:defT} we infer (up to a rotation) \eqref{eq:asympt_expans}, where
\[
A(r,\vartheta) := C \exp \left(\alpha (\vartheta + x_0) - \frac{2\alpha^2}{h}\log r \right).
\]
Finally, using again Proposition \ref{prop:target} we infer that, for some fixed $q$,
\[
q \leq \vartheta - \frac{2\alpha}{h}\log r \leq q + 2\pi + o(1)
\qquad\text{ as }r\to0^+,
\]
and also the second condition in \eqref{eq:gamma} follows.
\end{proof}

\textbf{Acknowledgements.} Work supported by the PRIN-2015KB9WPT Grant: ``Variational methods, with applications to problems in mathematical physics and geometry'', the ERC Advanced Grant 2013 n. 339958: ``Complex Patterns for Strongly Interacting Dynamical Systems -- COMPAT'', the ERC Advanced Grant 2013 n. 321186 ``ReaDi -- Reaction-Diffusion Equations, Propagation and Modelling'',
and by the INDAM-GNAMPA group.

\bibliography{spirals}
\bibliographystyle{abbrv}

\end{document}